\documentclass[reqno, a4paper, 10pt]{amsart}
\usepackage{amssymb,url, color}
\usepackage[colorlinks=true, bookmarks=true, pdfstartview=FitH, pagebackref=true, linktocpage=true]{hyperref}
\usepackage[short,nodayofweek]{datetime}
\usepackage{times}
\usepackage{float}
\usepackage{graphicx}
\usepackage{tikz}
\usetikzlibrary{patterns}
\usepackage{kpfonts, baskervald}

\usepackage[framemethod=TikZ]{mdframed}
\surroundwithmdframed[
 topline=false,
 bottomline=false,
 leftmargin=\parindent,
 skipabove=\medskipamount,
 skipbelow=\medskipamount
]{theorem}

\surroundwithmdframed[
 topline=false,
 bottomline=false,
 leftmargin=\parindent,
 skipabove=\medskipamount,
 skipbelow=\medskipamount
]{lemma}

\surroundwithmdframed[
 topline=false,
 bottomline=false,
 leftmargin=\parindent,
 skipabove=\medskipamount,
 skipbelow=\medskipamount
]{proposition}

\theoremstyle{plain}
\numberwithin{equation}{section}
\newtheorem{theorem}{Theorem}[section]
\newtheorem{lemma}[theorem]{Lemma}
\newtheorem{corollary}[theorem]{Corollary}
\newtheorem{proposition}[theorem]{Proposition}
\theoremstyle{remark}

\newtheorem{remark}[theorem]{Remark}

\hypersetup{
pdftitle={On the Hang--Yang conjecture for GJMS equations on $\S^n$},
pdfauthor={Ali Hyder, Qu\^oc Anh Ng\^o},
colorlinks = true,
linkcolor = magenta,
citecolor = blue,
}

\def\ve{\varepsilon}
\def\R{\mathbf R}

\newcommand{\gjms}{\, {\mathbf P}^{2m}_n}
\newcommand{\paneitz}{\, {\mathbf P}^4_3}
\newcommand{\Q}{\frac{n-2m}2 Q_n^{2m}}

\renewcommand{\S}{\mathbb S}

\newcommand{\nablaS}{\nabla_{g_{\S^3}}}
\newcommand{\DeltaS}{\Delta_{g_{\S^3}}}
\newcommand{\DSn}{\Delta_{g_{\S^n}}}
\newcommand{\vp}{\varphi}
\newcommand{\dsn}{d\mu_{\S^n}}
\newcommand{\dst}{d\mu_{\S^3}}

\numberwithin{equation}{section}

\makeatletter
\def\@cite#1#2{[\textbf{#1}\if@tempswa, #2\fi]}
\makeatother

\title[On the Hang--Yang conjecture for GJMS equations on $\S^n$]{On the Hang--Yang conjecture for GJMS equations on $\S^n$}

\def\cfac#1{\ifmmode\setbox7\hbox{$\accent"5E#1$}\else\setbox7\hbox{\accent"5E#1}\penalty 10000\relax\fi\raise 1\ht7\hbox{\lower1.05ex\hbox to 1\wd7{\hss\accent"13\hss}}\penalty 10000\hskip-1\wd7\penalty 10000\box7 }

\author[A. Hyder]{Ali Hyder}
\address[A. Hyder]{TIFR Centre for Applicable Mathematics, 
Sharadanagar, Bangalore 560065, India}
\email{\href{mailto: A. Hyder <hyder@tifrbng.res.in>}{hyder@tifrbng.res.in}}

\author[Q.A. Ng\^o]{Qu\'{\^o}c Anh Ng\^o}
\address[Q.A. Ng\^o]{
University of Science, Vietnam National University, Hanoi, Vietnam
\\
ORCID iD: 0000-0002-3550-9689}
\email{\href{mailto: Q.A. Ng\^o <nqanh@vnu.edu.vn>}{nqanh@vnu.edu.vn}}

\allowdisplaybreaks

\begin{document}


\begin{abstract}
This work concerns a Liouville type result for positive, smooth solution $v$ to the following higher-order equation
\[
\gjms (v) = \Q (\ve v+v^{-\alpha} ) 
\]
on $\S^n$ with $m \geq 2$, $3 \leq n < 2m $, $0<\alpha \leq (2m+n)/(2m-n)$, and $\ve>0$. Here $\gjms$ is the GJMS operator of order $2m$ on $\S^n$ and $Q_n^{2m} =(2/(n-2m)) \gjms(1)$ is constant. We show that if $\ve>0$ is small and $0<\alpha \leq (2m+n)/(2m-n)$, then any positive, smooth solution $v$ to the above equation must be constant. The same result remains valid if $\ve=0$ and $0<\alpha < (2m+n)/(2m-n)$. In the special case $n=3$, $m=2$, and $\alpha=7$, such Liouville type result was recently conjectured by F. Hang and P. Yang (\textit{Int. Math. Res. Not. IMRN}, 2020). As a by-product, we obtain the sharp (subcritical and critical) Sobolev inequalities 
\[
\Big( \int_{\S^n} v^{1-\alpha} \dsn \Big)^{\frac {2}{\alpha -1}} \int_{\S^n} v \gjms (v) \dsn
\geq \frac{\Gamma (n/2 + m)}{\Gamma (n/2 - m )} | \S^n|^\frac{\alpha + 1}{\alpha - 1}
\]
for the GJMS operator $\gjms$ on $\S^n$ under the conditions $n \geq 3$, $n=2m-1$, and $\alpha \in(0,1) \cup (1, 2n+1]$. A log-Sobolev type inequality, as the limiting case $\alpha=1$, is also presented.
\end{abstract}

\date{\bf \today \, at \currenttime}



\maketitle

\section{Introduction}

Let $n\geq 3$ be an odd integer, $2m>n$, and $0<\alpha\leq (n+2m)/(2m-n)$. In this work, we consider the following equation
\begin{subequations}\label{eq-MAIN}
\begin{align}
\gjms (v) = \Q (\ve v+v^{-\alpha} )\quad\text{in }\S^n. 
\tag*{ \eqref{eq-MAIN}$_\varepsilon$}
\end{align}
\end{subequations}
Here $\gjms$ is the well-known GJMS operator on $\S^n$ equipped with the standard metric $g_{\S^n}$, which is given as follows 
\[
\gjms :=\prod_{i=0}^{m-1}\left(-\DSn - (i+\frac n2)(i-\frac n2+1)\right),
\]
see \cite{GJMS}, and 
\[
Q_n^{2m} := \frac 2{n-2m} \gjms (1) = \frac 2{n-2m} \frac{\Gamma (n/2 + m)}{\Gamma (n/2 - m )}
\]
is a non-zero constant representing the so-called $Q$-curvature of $(\S^n, g_{\S^n})$, namely
\[
\gjms =(-\DSn)^m + \sum_{1\leq k \leq m-1} c_k (-\DSn)^k + \Q
\]
for suitable constants $c_k$ with $1 \leq k \leq m-1$. A special case of the operator $\gjms$, which has often been studied over the last two decades, is the well-known Paneitz operator, which is of fourth order. This example of a higher-order conformal operator gains interest because of its role in conformal geometry; see \cite{CGY02, HangYang16}. On $(\S^3, g_{\S^3})$, the Paneitz operator is given by 
\[
\paneitz = \DeltaS^2 + \frac 12 \DeltaS - \frac {15}{16},
\]
and therefore $Q_3^4 = -2 \Gamma (7/2)/\Gamma(-1/2)= 15/8$. Using the above recursive formula for $\gjms$ we can compute higher dimensional cases, for example
\[
\mathbf P_3^6 = -\DSn ^3 - \frac{23}4 \DSn^2 - \frac{27}{16}\DSn + \frac{315}{64} \quad \text{on } (\S^3, g_{\S^3})
\]
with $Q_3^6 = -105/32$ and
\[
\mathbf P_5^6 = -\DSn ^3 + \frac{13}4 \DSn^2 + \frac{93}{16}\DSn - \frac{945}{64} \quad \text{on } (\S^5, g_{\S^5})
\]
with $Q_5^6 =945/32$. One should pay attention on the sign difference of $Q_3^6$ and $Q_5^6$.

Our motivation of working on the equation \eqref{eq-MAIN}$_\varepsilon$ traces back to a recent conjecture by F. Hang and P. Yang in \cite{HangYang} that we are going to describe now. This conjecture concerns the following sharp critical Sobolev inequality on $\S^3$
\begin{equation}\label{eq-S}
 \|\phi^{-1}\|_{L^6(\S^3)}^{2}
\int_{\S^3} \Big[ (\DeltaS \phi)^2 - \frac 12 |\nablaS \phi |^2 - \frac {15}{16} \phi^2 \Big] \dst 
\geq -\frac {15}{16} | \S^3|^{4/3}
\end{equation}
for any $\phi \in H^2 (\S^3)$ with $\phi>0$, which was already proved in \cite{YangZhu2004} by symmetrization argument and in \cite{HangYang2004} by variational argument. Apparently, the inequality \eqref{eq-S} can be rewritten as follows
\begin{equation}\label{eq-S-new}
\|\phi^{-1}\|_{L^6(\S^3)}^{2} \int_{\S^3} \phi \paneitz (\phi) \dst \geq -\frac {15}{16} | \S^3|^{4/3}
\end{equation}
for any $0< \phi \in H^2 (\S^3)$, because the integral in \eqref{eq-S} is nothing but $\int_{\S^3} \phi \paneitz (\phi) \dst$. In \eqref{eq-S-new} and what follows, $|\S^n|$ denotes the surface area of $\S^n$. Besides, by Morrey's theorem, functions in $H^2(\S^3)$ are continuous and therefore the condition $\phi>0$ is understood in pointwise sense. By direct calculation, one can easily verify that equality in \eqref{eq-S-new} occurs if $\phi$ is any positive constant. This tells us that the Paneitz operator $\paneitz $ on the standard sphere $\S^3$ is no longer positive; see \cite{XuYang} for the assumption on the positivity of the Paneitz operator on closed $3$-manifolds.

In an effort to provide a new proof for \eqref{eq-S-new} with the sharp constant, the authors in \cite{HangYang} propose a new way to prove the above Sobolev inequality by considering the following minimizing problem
\begin{equation}\label{eq-EP}
\inf_{0<\phi \in H^2 (\S^3) } \|\phi^{-1}\|_{L^6(\S^3)}^{2} \Big[ \int_{\S^3} \phi \paneitz (\phi) \dst + \ve \int_{\S^3} \phi^2 \dst \Big]
\end{equation}
for small $\ve >0$. Thanks to the small perturbation $\ve \|\phi\|_{L^2(\S^3)}^2$, it is standard and straightforward to verify that the extremal problem \eqref{eq-EP} has a minimizer. Such a minimizer, denoted by $v_\ve$, eventually solves
\[
\paneitz (v_\ve ) + \ve v_\ve = - v_\ve^{-7}
\]
on $\S^3$, up to a constant. Here is the key observation: if the above equation only admits constant solution for small $\ve>0$, namely $v_\ve \equiv \text{const.}$, then one immediately has
\[
\|\phi^{-1}\|_{L^6(\S^3)}^{2} \Big[ \int_{\S^3} \phi \paneitz (\phi) \dst + \ve \int_{\S^3} \phi^2 \dst \Big]
\geq | \S^3|^{1/3} \Big[ \int_{\S^3} \paneitz (1) \dst + \ve |\S^3|\Big]
\]
for any $0<\phi \in H^2 (\S^3) $. Having this and as $\paneitz (1) = -(1/2) Q_3^4 = -15/16$, letting $\varepsilon \searrow 0$ yields \eqref{eq-S-new}. The novelty of this new approach is that it automatically implies the sharp form of \eqref{eq-S-new} with the precise sharp constant.

The above observation leads Hang and Yang to propose the following conjecture.

\smallskip
\noindent\textbf{The Hang--Yang conjecture} (\cite[page 3299]{HangYang}). \textit{Let $\ve > 0$ be a small number. If $v$ is a positive smooth solution to 
\[
\paneitz (v) + \ve v = - v^{-7}
\]
on $\S^3$, then $v$ must be a constant function.}

In a recent work Zhang \cite{Zhang} provides an affirmative answer to the above conjecture. The idea behind Zhang's proof is first to transfer the differential equation on $\S^3$ to some differential equation on $\R^3$ and then to classify solutions to that equation on $\R^3$. More precisely, let $\pi_N : \S^3 \to \R^3$ be the stereographic projection from the north pole $N$; see subsection \ref{subsec-SP} below. The pullback $(\pi_N^{-1})^*$ enjoys 
\[
(\pi_N^{-1})^* (g_{\S^3}) = \big( \frac 2{1+|x|^2} \big)^2 dx^2
\]
and for any smooth solution $v$ on $\S^3$ there holds
\[
\paneitz (v) \circ \pi_N^{-1} = \big( \frac 2{1+|x|^2} \big)^{-7/2} \Delta^2 \Big( \big( \frac 2{1+|x|^2} \big)^{-1/2} v \circ \pi_N^{-1} \Big).
\]
(Here and in the sequel, $\Delta$ is the usual Laplacian on Euclidean spaces.) Setting 
\begin{equation}\label{eq-u=v}
u (x) := \big( \frac {1+|x|^2} 2 \big)^{1/2} \big( v \circ \pi_N^{-1} \big) (x),
\end{equation}
we see that if $v$ solves $\paneitz (v ) + \ve v = - v^{-7}$ in $\S^3$, then $u$ solves
\[
\Delta^2 u(x) + \varepsilon \big( \frac 2{1+|x|^2} \big)^4 u (x) 
=- u^{-7} (x) 
\quad \text{in }\R^3.
\]
Via a dedicated argument based on the method of moving planes and techniques from potential theory, which are rather involved, it is proved that $u$ is radially symmetric. Finally, with the help of a Kazdan--Warner type identity, the function $v$ must be constant.

Inspired by the work of Zhang described above, we are interested in Hang--Yang's conjecture in higher dimensional cases, namely we want to seek for a suitable Liouville type result for positive, smooth solution to equations involving GJMS operators. This leads us to investigate solutions to \eqref{eq-MAIN}$_\varepsilon$. Very similar to situation studied by Hang and Yang, our motivation to study the equation \eqref{eq-MAIN}$_\varepsilon$ comes from the higher-order sharp critical Sobolev inequality; see Theorem \ref{thm-SS} below. Using the perturbation approach introduced in \cite{HangYang}, we are able  to establish a Liouville type result for solutions to \eqref{eq-MAIN}$_\ve$. 

Toward a suitable Liouville type result, let us first describe some preliminary results on \eqref{eq-MAIN}$_\ve$. Our first observation concerns the admissible range for $\ve$. As the perturbation approach is being used, we require the condition $\ve \geq 0$; see the proof of Lemma \ref{lem-SSexistence}. Now, by integrating both sides of \eqref{eq-MAIN}$_\ve$ over $\S^n$ and as $2m-n >0$ and $Q_n^{2m} \ne 0$ we conclude that
\[
 ( 1- \ve ) \int_{\S^n} v \dsn = \int_{\S^n} v^{-\alpha} \dsn.
\]
This immediately tells us that $\ve < 1$. Thus, the admissible range for $\ve$ is $0 \leq \varepsilon < 1$. Having this, let us now state the main result of this paper.
 
\begin{theorem} \label{thm-main}
Let $n\geq 3$ be odd and $m>n/2$. Then there exists $\ve_*\in (0,1 )$ such that under one of the following conditions
\begin{enumerate}
 \item either $\ve\in (0,\ve_*)$ and $0<\alpha\leq (n+2m)/(2m-n)$
 \item or $\ve=0$ and $0<\alpha< (n+2m)/(2m-n)$
\end{enumerate} 
any positive, smooth solution to \eqref{eq-MAIN}$_\ve$ must be constant.
\end{theorem} 

We have the following remarks:
\begin{itemize}
 \item The above result again confirms the Hang--Yang conjecture for the Paneitz operator on $\S^3$, and generalizes the result of Zhang in the critical setting in higher dimensional cases.
 \item Theorem \ref{thm-main} can be compared with the Liouville type results obtained by Bidaut-V\'eron and V\'eron in \cite[Theorem 6.1]{VV} for the Emden equation, see also the work of Gidas and Spruck in \cite{Gidas}. Note that the condition $\alpha<(n+2m)/(2m-n)$ is sharp for $\ve=0$ as the result does not hold if $\alpha = (n+2m)/(2m-n)$. This is because in this limiting case the equation \eqref{eq-MAIN}$_0$ is conformally invariant; see section \ref{sec-compact}.
 \item The threshold $\ve_*$ is given in Lemma \ref{lem-ve_*}.
 \item Although for any $0 \leq \varepsilon < 1$, equation \eqref{eq-MAIN}$_\ve$ always admits the trivial solution $v_\ve\equiv (1 - \ve)^{-1/(\alpha+1)}$, but it is not clear whether or not the above Liouville type result still holds for $\ve \in [\ve_*, 1)$. This seems to be an interesting open question.
\end{itemize}

To prove Theorem \ref{thm-main}, we adopt the strategy used by Zhang. Such strategy can be formulated as the following two main steps: first to transfer \eqref{eq-MAIN}$_\ve$ in $\S^n$ to the  equation \eqref{eq-MAIN-DE}$_\ve$ and the corresponding integral equation in $\R^n$, then to study symmetry properties of solutions to these equations for small $\ve>0$. However, to be able to handle higher-order cases, our approach is significantly different from Zhang. One major reason is that less results is known for the higher-order cases compared to the case $m=2$. For example, we do not know if the preliminary results of Hang and Yang mentioned in \cite[section 2]{Zhang} are available for $m \geq 3$. Because of this difficulty, instead of the differential equation \eqref{eq-MAIN-DE}$_\ve$, we mainly work on the corresponding integral equation on $\R^n$, and directly prove compactness results and symmetry properties of solutions. As pays off, our analysis is much simpler, and could handle higher-order cases efficiently. 

As the operator $\gjms$ is conformally covariant, for any smooth function $\vp$ on $\S^n$ we have the following identity ($\pi$ denotes the stereographic projection from $\S^n$ to $\R^n$ with respect to either the north or the south pole)
\[
\gjms (\vp)\circ\pi^{-1}=\big( \frac{2}{1+|x|^2}\big)^{-\frac{n+2m}{2}} (-\Delta)^m\Big(\big( \frac{2}{1+|x|^2}\big)^\frac{n-2m}{2} \vp\circ\pi^{-1}\Big);
\]
see e.g. \cite[Section 2]{Hang}. Then, similar to \eqref{eq-u=v}, by setting 
\begin{align}\label{def-u}
u(x):=\big( \frac{2}{1+|x|^2}\big)^\frac{n-2m}{2} \big( v\circ\pi^{-1} \big)
\end{align} 
and
\begin{align}\label{def-F}
F_{\ve,u}(x):=\ve \big( \frac{2}{1+|x|^2}\big)^ {2m} u(x) +\big( \frac{2}{1+|x|^2}\big)^{\frac{n+2m}{2} +\alpha \frac{n-2m}{2} } u(x)^{-\alpha}
\end{align}
we see that $u$ satisfies 
\begin{subequations}\label{eq-MAIN-DE}
\begin{align}
(-\Delta)^m u=\Q F_{\ve,u} \quad \text{in } \R^n.
\tag*{ \eqref{eq-MAIN-DE}$_\varepsilon$}
\end{align}
\end{subequations}
In view of \eqref{def-u}, we know that the function $u$ on $\R^n$ has exact growth $|x|^{2m-n}$ at infinity. This additional information allows us to transfer the differential equation \eqref{eq-MAIN-DE}$_\ve$ into the following integral equation
\[
u(x) = \gamma_{2m,n} \int_{\R^n} |x-y|^{2m-n} F_{\ve,u} (y) dy \quad \text{on } \R^n
\]
for some constant $\gamma_{2m,n}>0$; see Theorem \ref{thm-PDE-IE} below. Notice that in general there might be more solutions to \eqref{eq-MAIN-DE}$_\ve$ than the above integral equation, see e.g. \cite{HW} and \cite{DuocNgo22}. 

Let us emphasize that transferring to an equivalent integral equation on $\R^n$ also appears in the work of Zhang, but the proof provided in \cite{Zhang} does not seem to work in our case. Similar integral representation in the fractional setting also appears in \cite{FKT22}. In our work, by exploiting some nice structures on $\S^n$ as well as some intriguing properties of the stereographic projection, we offer a completely new argument, which is surprisingly simpler; see section \ref{sec-DE->IE}.

Having the above integral equation in hand, we use a variant of the method of moving planes in the integral form to show that any positive smooth solution $u$ to the above integral equation with exact growth $|x|^{2m-n}$ at infinity must be radially symmetric. The symmetry of solutions to the integral equation helps us to conclude that the corresponding function $v$, appeared as in \eqref{def-u}, must be constant. The strategy we just describe seems to be very simple and straightforward at the first glance, but there are two major difficulties that we want to highlight. First, it is worth emphasizing that the method of moving planes and its variants work well in the case of equations with positive exponents; unfortunately, our equations, both differential and integral forms, have a negative exponent. Second, by analyzing the form of $F_{\ve,u}$ in \eqref{def-F}, one immediately notices that because of our special choice of perturbation, there are two powers of $u$, whose exponents have opposite sign. Unless $\ve = 0$, otherwise to run the method of moving planes, one needs to establish certain compactness result for solutions to \eqref{eq-MAIN}$_\ve$ for suitable small $\ve$, which costs us some energy. 

Concerning classification of solutions to \eqref{eq-MAIN-DE}$_\ve$ with $\ve=0$ and with the RHS depending only on $u$, that is equation of the form $(-\Delta)^mu=cu^{-\alpha}$ we refer to \cite{HW,Ngo,Li} and the references therein.

Finally, to illustrate our finding on a Liouville type result for solutions to \eqref{eq-MAIN}$_\ve$, we revisit the sharp critical Sobolev inequality for $\gjms$ on $\S^n$ proved in \cite{Hang}. In fact, we offer both critical and subcritical inequalities at once.

\begin{theorem}\label{thm-SS}
Let $n \geq 3$ be an odd integer and $m=(n+1)/2$. Then, for any $\phi \in H^m (\S^n)$ with $\phi>0$ and any $\alpha \in (0,1) \cup (1, 2n+1]$, we have the following sharp Sobolev inequality
\begin{equation}\label{eq-sSobolev}
 \Big( \int_{\S^n} \phi^{1-\alpha} \dsn \Big)^{\frac {2}{\alpha -1}} \int_{\S^n} \phi \gjms (\phi) \dsn
\geq \frac{\Gamma (n/2 + m)}{\Gamma (n/2 - m )} | \S^n|^\frac{\alpha + 1}{\alpha - 1}.
\end{equation} 
Moreover, the equality occurs if $\phi$ is any positive constant.
\end{theorem}

Let us have some comments on Theorem \ref{thm-main} above.

\begin{remark}\label{rmk}\ \
\begin{itemize}
 \item Although the condition $n=2m-1$ is not required in Theorem \ref{thm-main}, but in our proof of \eqref{eq-sSobolev} we heavily use it as in this case we have the advantage of $Q$-curvature $Q^{2m}_n$ being positive. In general, the inequality \eqref{eq-sSobolev} is not true for $n<2m-3$, see e.g. \cite{FKT22}. 
 
 \item Apparently, by chosing $\alpha = (n+2m)/(2m-n)=2n+1$, our inequality \eqref{eq-sSobolev} includes the following critical Sobolev inequality
\begin{equation}\label{eq-cSobolev}
\Big( \int_{\S^n} \phi^{-\frac{2n}{2m-n}} \dsn \Big)^{\frac {2m-n}{n}} \int_{\S^n} \phi \gjms (\phi) \dsn
\geq \frac{\Gamma (n/2 + m)}{\Gamma (n/2 - m )} | \S^n|^\frac{2m}{n} ,
\end{equation}
which was already proved in \cite{Hang}, see also \cite{HangYang2004} and \cite{FKT22}.

 \item  The case $\alpha = 1$ is excluded in Theorem \ref{thm-SS} due to the presence of the term $1/(\alpha-1)$. For $\alpha=1$, by a limiting argument one obtains the  inequality \eqref{eq-sSobolev-limit} below.

 \item Our last comment concerns the order of the inequality \eqref{eq-sSobolev} as $\alpha$ varies. It turns out that the subcritical case $0<\alpha<2n+1$ can be obtained from the critical case $\alpha=2n+1$, see Section \ref{sec-SS} for more details.

\end{itemize}
\end{remark}

Note that our inequality \eqref{eq-sSobolev} can be rewritten as
\[ 
 \Big( \fint_{\S^n} \phi^{1-\alpha} \dsn \Big)^{\frac {2}{\alpha -1}} \fint_{\S^n} \phi \gjms (\phi) \dsn
\geq \frac{\Gamma (n/2 + m)}{\Gamma (n/2 - m )},
\]
where $ \fint_{\S^n} := |\S^n|^{-1} \int_{\S^n}$ denotes the average. Using this new form one can easily compute the limit as $\alpha \searrow 1$ to obtain an inequality in the limiting case as shown in the following corollary.

\begin{corollary}\label{cor-SS-limit}
Let $n \geq 3$ be an odd integer and $m=(n+1)/2$. Then, for any $\phi \in H^m (\S^n)$ with $\phi>0$, we have the following sharp Sobolev inequality
\begin{equation}\label{eq-sSobolev-limit}
\exp \Big(-2\fint_{\S^n}\log\phi \dsn \Big) \fint_{\S^n}\phi \gjms (\phi)\dsn\geq \frac{\Gamma(n/2+m)}{\Gamma(n/2-m)}.
\end{equation} 
Moreover, the equality occurs if $\phi$ is any positive constant.
\end{corollary}

It turns out that without using any limit process, one can still obtain \eqref{eq-sSobolev-limit} directly from \eqref{eq-sSobolev}; see Proposition \ref{prop-order}. As such, we omit the proof of \eqref{eq-sSobolev-limit}. Without using averages, \eqref{eq-sSobolev-limit} can be rewritten as follows
\[
\exp \Big(- \frac 2{|\S^n|}\int_{\S^n}\log\phi \dsn \Big) \int_{\S^n}\phi \gjms (\phi)\dsn\geq \frac{\Gamma(n/2+m)}{\Gamma(n/2-m)} |\S^n|.
\]
To the best of our knowledge, the above inequality (or the inequality \eqref{eq-sSobolev-limit}) seems to be new. 

Our final comment concerns a possible generalization to the fractional setting. Indeed, it seems that part of our argument can be quickly extended to the case of fractional operators of order $2s > n$ instead of GJMS operators of integer order $2m>n$. However, to maintain our work in a reasonable length, we leave this future research. 

Before closing this section, let us mention the organization of the paper.

\tableofcontents

 
\section{Some auxiliary results} 
\label{sec-DE->IE}


\subsection{Basics of the stereographic projection}
\label{subsec-SP}

As routine, we denote by $\pi_N$ and $\pi_S$ the stereographic projections from the north pole $N$ and from the south pole $S$ of the sphere $\S^n $ respectively. If we denote by $(x,x_{n+1})$ a general point in $\R^{n+1} = \R^n \times \R$, then we have the following expressions for $\pi_N$
\[
\pi_N (x, x_{n+1}) = \frac x{1-x_{n+1}}, \quad
\pi_N^{-1} ( x ) = \Big( \frac{2x}{|x|^2 + 1}, \frac{|x|^2 - 1}{|x|^2 +1}\Big) .
\]
Likewise, we also have similar expressions for $\pi_S$. But these expressions for $\pi_S$ can be derived quickly from those for $\pi_N$ by changing the sign of the last coordinate. In this sense, we arrive at
\[
\pi_S (x, x_{n+1}) = \frac x{1 + x_{n+1}}, \quad
\pi_S^{-1} ( x ) = \Big( \frac{2x}{|x|^2 + 1}, -\frac{|x|^2 - 1}{|x|^2 +1}\Big) .
\]
The following observation plays some role in our analysis.

\begin{lemma}\label{lem-Pi_N=Pi_S}
There holds
\[
\pi_N^{-1} (x) = \pi_S^{-1} \big( \frac x{|x|^2} \big), \quad 
\pi_S^{-1} (x) = \pi_N^{-1} \big( \frac x{|x|^2} \big)
\]
in $\R^n \setminus \{ 0 \}$.
\end{lemma}

\begin{proof}
These identities follows from the above expressions for $\pi_N$ and $\pi_S$. 

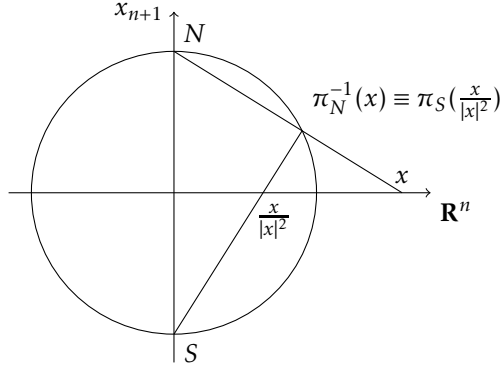
\begin{figure}[H]
\begin{tikzpicture}[scale=0.75]
\draw[->] (-2.9, 0) -- (4.5,0) node[anchor =north west] {$\R^n$};
\draw[->] (0,-3) -- (0,3.2) node[anchor =east] {$x_{n+1}$};
\draw[-] (0,0) circle (2.5);
\draw[-] (0, 2.5) node[anchor = south west] {$N$} -- (4,0) node[anchor =south] {$x$} ;
\draw[-] (0, -2.5) node[anchor = north west] {$S$} -- (2.25,1.1) node[anchor =south west] {$\pi_N ^{-1}(x) \equiv \pi_S (\frac x{|x|^2})$} ;
\node at (1.75,-0.5) {$\frac {x}{|x|^2}$};
\end{tikzpicture}
\caption[]{Relation between $\pi_N^{-1}$ and $\pi_S$.}\label{fig-SP}
\end{figure}

We leave the details for interested readers; also see Figure \ref{fig-SP} above.
\end{proof}


\subsection{From differential equations to integral equations}

Let $v$ be a positive, smooth solution to \eqref{eq-MAIN}$_\ve$. Recall from \eqref{eq-MAIN-DE}$_\varepsilon$ that the projected function $u$, defined by \eqref{def-u}, solves
\[
(-\Delta)^m u=\Q F_{\ve,u} \quad \text{in } \R^n.
\]
The main result of this subsection is to show that $u$ actually solves the corresponding integral equation \eqref{IEn}. To achieve this goal, we need certain preparation including the introduction of a uniform constant that we are going to describe now.

Since $n$ is an odd integer, for some dimensional constant $c_{2m,n} \ne 0$ we have 
\[
(-\Delta)^m \big( c_{2m,n}|x|^{2m-n} \big) =\delta_0,
\]
where $\delta_0$ is the Dirac measure at the origin. For convenience, we also set
\[
\gamma_{2m,n}:=c_{2m,n}\Q .
\]
For simplicity, throughout the paper, we often denote by $C$ a generic constant whose value could vary from estimate to estimate. We now state our main result in this subsection.

\begin{theorem} \label{thm-PDE-IE} 
We have \[ \gamma_{2m,n}>0 \] and 
\begin{align} \label{IEn} 
u(x)=\gamma_{2m,n}\int_{\R^n} |x-y|^{2m-n} F_{\ve,u}(y)dy
\end{align} 
where $F_{\ve, u}$ is given by \eqref{def-F}.
\end{theorem}

Notice that the integral in \eqref{IEn} is well-defined everywhere in $\R^n$. Indeed, as $v$ is positive everywhere on $\S^n$, we have from \eqref{def-u} that $u(x)\approx |x|^{2m-n}$ for $|x|\gg 1$, and hence 
\begin{align} \label{est-infty}
(1+|x|^{2m-n})F_{\ve,u}(x)\leq \frac{C}{1+|x|^{2n}}.
\end{align}
In order to prove the above theorem we define the following functions associated with the projections $\pi_N$ and $\pi_S$: 
\[
u_N (x): = \big( \frac {1+|x|^2}2 \big)^{\frac {2m-n}2} (v \circ \pi_N^{-1}) (x)
\]
and
\[
u_S (x) := \big( \frac {1+|x|^2}2 \big)^{\frac {2m-n}2} (v \circ \pi_S^{-1}) (x)
\]
in $\R^n$. In view of the integral equation \eqref{IEn}, we denote
\[
\widetilde u_N (x) := \gamma_{2m,n} \int_{\R^n} |x-y|^{2m-n} F_{\ve, u_N}(y) dy
\]
and
\[
\widetilde u_S (x) := \gamma_{2m,n} \int_{\R^n} |x-y|^{2m-n} F_{\ve, u_S} (y) dy
\]
in $\R^n$. Our aim is to show that $u_N \equiv \widetilde u_N$ and that $\gamma_{2m,n}>0$. This will be done through several steps. Our first observation is as follows.

\begin{lemma}\label{lem-uNuS*}
We have
\[
u_S (x) = |x|^{2m-n} u_N \big( \frac x{|x|^2} \big) , \quad
u_N (x) = |x|^{2m-n} u_S \big( \frac x{|x|^2} \big)
\]
in $\R^n\setminus\{0\}$.
\end{lemma}

\begin{proof}
This is elementary. Indeed, let us compute $u_S$. Clearly, with the help of Lemma \ref{lem-Pi_N=Pi_S}, we have
\begin{align*}
u_S (x) &= \big( \frac {1+|x|^2}2 \big)^{\frac {2m-n}2} v \big(\pi_S^{-1} (x) \big)\\
&=\big( \frac {1+|x|^2}2 \big)^{\frac {2m-n}2} v \big( \pi_N^{-1} \big( \frac x{|x|^2} \big) \big)\\
&=|x| ^{2m-n}\big( \frac {1+|x/|x|^2|^2}2 \big)^{\frac {2m-n}2} v \big( \pi_N^{-1} \big( \frac x{|x|^2} \big) \big) \\
& = |x|^{2m-n} u_N \big( \frac x{|x|^2} \big) ,
\end{align*}
which gives the desired formula for $u_S$. The identity for $u_N$ can be verified similarly.
\end{proof}

Our next observation is similar to that in Lemma \ref{lem-uNuS*}.

\begin{lemma}\label{lem-tuNtuS*}
We have 
\[
\widetilde u_S (x) = |x|^{2m-n} \widetilde u_N \big (\frac x{|x|^2} \big) , \quad
\widetilde u_N(x) = |x|^{2m-n} \widetilde u_S \big (\frac x{|x|^2} \big)
\]
in $\R^n \setminus \{ 0\}$.
\end{lemma}

\begin{proof}
This is also elementary but rather involved. Indeed, let us verify the first identity. With a change of variable $y=z/|z|^2$ and by  Lemma \ref{lem-uNuS*} we easily get
\begin{align*}
|x| ^{2m-n}\widetilde u_N \big (\frac x{|x|^2} \big)
&= \gamma_{2m,n}|x|^{2m-n}\int_{\R^n} \Big| \frac x{|x|^2} -y \Big| ^{2m-n} F_{\ve, u_N}(y) dy \\
&= \gamma_{2m,n}|x|^{2m-n}\int_{\R^n} \Big| \frac x{|x|^2} -\frac{z}{|z|^2} \Big| ^{2m-n} F_{\ve, u_N}(\frac{z}{|z|^2}) \frac{dz}{|z|^{2n}} \\
&= \gamma_{2m,n} \int_{\R^n} | x -z | ^{2m-n} F_{\ve, u_S}(z) dz \\
 &= \widetilde u_S (x),
\end{align*} where in the second last equality we have used the following facts:
\[
\left|\frac{x}{|x|^2}-\frac{z}{|z|^2} \right|=\frac{|x-z|}{|x||z|},\quad F_{\ve, u_N}(\frac{z}{|z|^2})=|z|^{2m+n}F_{\ve, u_S}(z).
\]
The second identity can be verified similarly.
\end{proof}

Now we are able to examine $u_N -\widetilde u_N$ and $u_S - \widetilde u_S$.

\begin{lemma}
The following functions 
\[
P_N := u_N - \widetilde u_N, \quad P_S := u_S - \widetilde u_S
\]
are polynomials in $\R^n$ of degree at most $2m-n$.
\end{lemma}

\begin{proof} 
Before proving, we see that both $P_N$ and $P_S$ are well-defined everywhere in $\R^n$. Now it follows from \eqref{est-infty} that the function $\widetilde u_N$ satisfies 
\[
\widetilde u_N(x)\leq C(1+|x|^{2m-n})\quad\text{for }x\in\R^n.
\]
This together with the growth of $u_N$ implies that $|P_N(x)|\leq C(1+|x|^{2m-n}) $. Since 
\[
\Delta^m P_N = \Delta^m u_N - \Delta^m \widetilde u_N = 0,
\]
we conclude that $P_N$ is a polynomial in $\R^n$ of degree at most $2m-n$; see \cite[Theorem 5]{Martinazzi}. A similar argument applies to $P_S$ yielding the same conclusion for $P_S$.
\end{proof}

Finally, we are in a position to prove Theorem \ref{thm-PDE-IE}, which simply follows from the next two lemmas. 

\begin{lemma}
There hold $u_N \equiv \widetilde u_N$ and $u_S \equiv \widetilde u_S$ everywhere. 
\end{lemma}

\begin{proof}
As
\[
u_S (x) = |x|^{2m-n} u_N \big( \frac x{|x|^2} \big) , \quad 
\widetilde u_S (x) = |x|^{2m-n} \widetilde u_N \big (\frac x{|x|^2} \big) 
\]
we obtain
\begin{align*}
P_S (x) = |x|^{2m-n} P_N \big( \frac x{|x|^2} \big) ,
\end{align*}
which is a polynomial (of degree at most $2m-n$). Surely, as $n$ is odd, this is impossible because $|x|^{2m-n}$ cannot be a polynomial unless $P_N \equiv P_S \equiv 0$, which implies that $u_N \equiv \widetilde u_N$ and $u_S \equiv \widetilde u_S$. This completes the proof.
\end{proof}

\begin{lemma}
There hold $\gamma_{2m,n}>0$.
\end{lemma}

\begin{proof}
The claim $\gamma_{2m,n}>0$ follows trivially by seeing its definition
\[
\gamma_{2m,n} = c_{2m,n} Q_n^{2m}.
\]
Note that $Q_n^{2m} >0$ and that $c_{2m,n}>0$ because $n<2m$ and $n$ is odd. However, the claim can also be seen from the fact that $v\equiv 1$ is a solution to \eqref{eq-MAIN}$_0$. More precisely, making use of $v \equiv 1$ and \eqref{def-u} one has the following identity
\[
\big( \frac{2}{1+|x|^2}\big)^\frac{n-2m}{2} =\gamma_{2m,n}\int_{\R^n} |x-y|^{2m-n}\big( \frac{2}{1+|y|^2}\big)^\frac{n+2m}{2} dy
\]
everywhere in $\R^n$.
\end{proof}

We conclude this subsection by noting that our approach to prove Theorem \ref{thm-PDE-IE} can be used for the case of equations with positive exponent. For example, without using any super polyharmonic property, as in \cite{CLS}, our new approach offers a very simple and straightforward proof to convert differential equations on $\S^n$ to the corresponding integral equations on $\R^n$, detail will appear elsewhere.


\subsection{Pohozaev-type identity}
 
Our last auxiliary result is a Pohozaev-type identity, which shall be used in the proof of a compactness type result; see section \ref{sec-compact} below. For simplicity, we let
\begin{equation}\label{c_alpha}
c_\alpha:=\alpha \frac{2m-n}2 - \frac{2m+n}{2} \leq 0.
\end{equation} 
 
For future usage, let us state our Pohozaev-type identity in a more general framework.

\begin{lemma}\label{Poho-Rn} 
Let $Q\in C^1(\R^n)$ be such that 
\[
|Q(x)|\lesssim (1+|x|)^{-n+(\alpha-1)(2m-n) -\delta},
\]
for some $\delta>0$. Let $u$ be a positive, regular solution to 
\begin{align}\label{int-u-Q}
u(x)=\int_{\R^n}|x-y|^{2m-n}Q(y)u^{-\alpha}(y)dy,
\end{align}
where $u$ satisfies 
\[
u\gtrsim (1+|x|)^{2m-n} \quad \text{if } \alpha>1
\] 
and that 
\[
u\lesssim (1+|x|)^{2m-n} \quad \text{if } 0<\alpha<1.
\] 
Then, for $\alpha\neq 1$, there holds
\[
\int_{\R^n} (x\cdot\nabla Q)u^{1-\alpha}dx=c_\alpha \int_{\R^n}Q u^{1-\alpha}dx,
\] 
provided $(x\cdot\nabla Q)u^{1-\alpha}\in L^1(\R^n)$. 
\end{lemma}

\begin{proof} 
The proof given below is more or less standard. As $x=(1/2)(x+y+x-y)$ and
\[
\nabla_x (|x-y|^{2m-n}) = (2m-n) |x-y|^{2m-n-2} (x-y),
\]
by differentiating under the integral sign in \eqref{int-u-Q}, we obtain 
\[
x\cdot\nabla u(x)=\frac{2m-n}{2} u(x)+\frac{2m-n}{2} \int_{\R^n}\frac{|x|^2-|y|^2}{|x-y|^{n+2-2m}}Q(y)u^{-\alpha}(y)dy.
\]
Multiplying the above identity by $Q(x)u^{-\alpha}(x)$, and then integrating the resultant on $B_R$ we arrive at
\begin{align*}
\frac{1}{1-\alpha}\int_{B_R}Q & ( x \cdot \nabla u^{1-\alpha} ) dx =\frac{2m-n}2\int_{B_R} Qu^{1-\alpha} \\
&+\frac{2m-n}2 \int_{B_R}Q(x)u^{-\alpha}(x) \Big( \int_{\R^n}\frac{|x|^2-|y|^2}{|x-y|^{n+2-2m}}Q(y)u^{-\alpha}(y)dy \Big) dx.
\end{align*} 
Integration by parts leads to 
\begin{align*}
\int_{B_R}Q ( x \cdot \nabla u^{1-\alpha} ) dx =& -\int_{B_R} ( x \cdot \nabla Q ) u^{1-\alpha} dx - n\int_{B_R} Q u^{1-\alpha} dx\\
& + R\int_{\partial B_R} Q u^{1-\alpha} dx.
\end{align*} 
Hence,
\begin{equation}\label{Pohozaev-final}
\begin{aligned} 
\frac{R}{1-\alpha}\int_{\partial B_R}& Q u^{1-\alpha}d\sigma - 
\frac{2m-n}2\int_{B_R}\int_{\R^n}\frac{|x|^2-|y|^2}{|x-y|^{n+2-2m}}Q(y)u^{-\alpha}(y)Q(x)u^{-\alpha}(x)dy dx\\
&=\frac 1{1-\alpha} \Big[ \frac{(2m+n)-\alpha(2m-n)}{2} \int_{B_R} Qu^{1-\alpha}dx+ \int_{B_R} (x\cdot\nabla Q)u^{1-\alpha}dx \Big].
\end{aligned} 
\end{equation}
Thanks to the decay assumption on $Q$ and the growth of $u$, we easily get
\[
\lim_{R\to \infty} \Big( R \int_{\partial B_R}Q u^{1-\alpha}d\sigma \Big) = 0,
\]
and clearly
\[
\int_{\R^n}\int_{\R^n}\frac{|x|^2-|y|^2}{|x-y|^{n+2-2m}}Q(y)u^{-\alpha}(y)Q(x)u^{-\alpha}(x)dy dx =0
\] 
due to the antisymmetry of the integrand. Furthermore, under the assumptions on $Q$ and on $u$, there holds $Q u^{1-\alpha} \in L^1 (\R^n)$. Hence, by sending $R \nearrow +\infty$, we conclude that the LHS of \eqref{Pohozaev-final} vanishes, giving the desired identity. This completes the proof.
\end{proof}

Let us now discuss how to use our Pohozaev-type identity in the current setting. Recall that the solution $v$ to \eqref{eq-MAIN}$_\varepsilon$ is positive and smooth on $\S^n$. Thanks to \eqref{def-u} we deduce that $u$ enjoys the upper and lower growths as in Lemma \ref{Poho-Rn}. Hence, we have a Pohozaev-type identity for $u$ whenever $\alpha \ne 1$. We shall use this identity in the proof of Lemma \ref{lem-lower} below.


\section{Compactness results}
\label{sec-compact}
 
This section is devoted to a compactness type result for solutions to \eqref{eq-MAIN}$_\ve$, which is of interest itself; see Theorem \ref{thm-uniform} below. Heuristically, one should study the compactness result for fixed $\ve$ and $\alpha$. However, to derive useful estimates for our analysis, one needs certain compactness result which is independent of $\ve$; see the proof of Lemmas \ref{lem-lambda0} and \ref{lem-ve_*} below.
 
\begin{theorem}\label{thm-uniform} 
Let $\ve^* \in (0,1)$ and $\alpha\in (0,(2m+n)/(2m-n)]$ be arbitrary but fixed. Assume that $v_k=v_{\ve_k}$ is a sequence of positive regular solutions to \eqref{eq-MAIN}$_{\ve_k}$ for some $\ve_k\in (0,\ve^*)$. Then there exists $C=C(\ve^*)>0$ such that 
\[
\frac 1C \leq v_k \leq C \quad \text{in } \S^n
\]
for all $k$. The same conclusion holds true for $\ve_k \in [0, \ve^*)$ if $\alpha \in (0, (2m+n)/(2m-n))$.
\end{theorem}

It is worth noting that the above compactness fails for solutions to \eqref{eq-MAIN}$_0$ in the case $\alpha=(n+2m)/(2m-n)$ due to the conformally invariant property of the underlying equation. More specifically, fixing any solution $v$ to
\[
\gjms (v) = \Q v^{\frac{n+2m}{2m-n}} \quad \text{in } \S^n
\]
and let 
\[
v_\phi = (v \circ \phi) |\det (d\phi)|^{-\frac 1{2n}},
\]
where $\phi$ is any conformal transformation on $\S^n$. Then, it is well-known that $v_\phi$ solves the same equation in $\S^n$. Hence, if one choose a sequence of $\phi$ in such a way that $|\det (d\phi)| \searrow 0$, then the sequence $v_\phi$ is unbounded in $\S^n$. 

In order to prove the above theorem we first need to rule out the possibility that the sequence $v_k$ will eventually touch zero. This in particular implies the lower estimate in the theorem.

\begin{lemma}\label{lem-lower} 
Under the hypothesis of Theorem \ref{thm-uniform} we have 
\[
\inf_{k\geq1}\min_{\S^n}v_k>0.
\] 
\end{lemma}
\begin{proof}

We assume by contradiction that the lemma is false. Then, up to a subsequence, we assume that
\[
\min_{\S^3}v_k\to0 \quad \text{as } k \to \infty.
\]
Without loss of generality we can further assume that the minimum of $v_k$ is attained at the south pole. Let $u_k$ be defined by \eqref{def-u} using $\pi_N$, and let $ F_k:=F_{\ve_k,u_k}$ as in \eqref{def-F}. In view of \eqref{def-u} and $2m>n$, the function $u_k$ achieves its minimum at $0$. By Theorem \ref{thm-PDE-IE}, the function $u_k$ satisfies 
\begin{align}\label{int-uk} 
u_k(x) = \gamma_{2m,n}\int_{\R^n} |x-y|^{2m-n} F_k(y)dy .
\end{align}
To show that this is also not the case, we use the Pohozaev-type identity of  Lemma \ref{Poho-Rn} and the role played by $\ve_k$ and $\alpha$. Indeed, as $F_k>0$ we first obtain
\begin{equation}\label{est-bound-yF_k}
u_k(0) = \gamma_{2m,n}\int_{\R^n} |y|^{2m-n} F_k(y)dy = o(1)_{k \to \infty}.
\end{equation}
Using this one can show that
\begin{equation}\label{est-limit-u_k}
\lim_{k\to\infty} u_k (x) = \infty \quad \text{for each } x \in \R^n \setminus \{0\}.
\end{equation}
Indeed, by way of contradiction suppose that there is some $x_0 \in \R^n \setminus \{0\}$ such that $u_k(x_0) = O(1)_{k\to \infty}$. As
\begin{align*}
\frac{u_k (x_0)}{ \gamma_{2m,n}} &=\int_{\R^n} |x_0 - y|^{2m-n} F_k(y)dy \\
&\geq 2^{-2m+n+1} \int_{\R^n} |x_0|^{2m-n} F_k(y)dy - \int_{\R^n} |y|^{2m-n} F_k(y)dy 
\end{align*}
we obtain
\[
\int_{\R^n} F_k(y)dy = O(1)_{k\to\infty},
\] 
thanks to $u_k (0) = O(1)_{k \to \infty}$. Hence
\begin{equation}\label{est-bound-(1+y)F_k}
\int_{\R^n} (1+|y|^{2m-n}) F_k(y)dy = O(1)_{k\to\infty}.
\end{equation}
Consequently, for any $x \in \R^n$, one can estimate
\begin{align*}
\frac{u_k(x)}{ \gamma_{2m,n}} &= \int_{\R^n} |x-y|^{2m-n} F_k(y)dy
\leq 2^{2m-n - 1}\int_{\R^n} (|x|^{2m-n}+|y|^{2m-n}) F_k(y)dy,
\end{align*}
which leads to
\[
u_k(x) \leq C(1+|x|^{2m-n}) \quad \text{in } \R^n
\]
for some constant $C>0$. Having this, one can bound $F_k$ from below near the origin. For example, for any $x \in B_2$,
we easily get
\begin{align*}
F_k (x) & \geq \big( \frac{2}{1+|x|^2}\big)^{-c_\alpha } u_k(x)^{-\alpha}
 \geq \frac 1{C^\alpha} \big( \frac{2}{1+|x|^2}\big)^{\frac{n+2m}{2} } 
\geq \frac 1{C^\alpha} \Big(\frac{2}{5}\Big)^{\frac{n+2m}{2} } ,
\end{align*}
thanks to $u_k(x) \leq C(1+|x|^2)^{(2m-n)/2}$ in $\R^n$. However, this violates the fact that $u_k (0) = o(1)_{k \to \infty}$. Indeed, 
\begin{align*}
\frac{u_k(0)}{ \gamma_{2m,n}} &\geq \int_{B_2 \setminus B_1} |y|^{2m-n} F_k(y)dy 
 \geq \frac 1{C^\alpha} \Big(\frac{2}{5}\Big)^{\frac{n+2m}{2} } \int_{B_2 \setminus B_1} |y|^{2m-n} dy > 0
\end{align*}
for all $k$. Thus, no such a point $x_0$ could exist, and hence \eqref{est-limit-u_k} must hold. Notice that the above proof also reveals the fact that
\begin{equation}\label{est-limit-F_k}
\lim_{k\to\infty}\int_{\R^n} F_k(y)dy= \infty,
\end{equation}
otherwise by \eqref{est-bound-yF_k} one would again have \eqref{est-bound-(1+y)F_k} and again this leads to a contradiction. Now we normalize $u_k$ and $F_k$ as follows
\[
\widetilde u_k:=\frac{u_k}{\gamma_{2m,n}\int_{\R^n} F_kdy },\quad \widetilde F_k:=\frac{ F_k}{\int_{\R^n} F_kdy} .
\]
Then 
\[
\widetilde u_k(x)=\int_{\R^n}|x-y|^{2m-n}\widetilde F_k(y)dy,\quad \int_{\R^n}\widetilde F_kdy=1.
\]
Having \eqref{est-limit-F_k}, it is clear that $\widetilde u_k(0)\to0$ and 
\[
|\nabla\widetilde u_k(x)| \leq (2m-n) \int_{\R^n}|x-y|^{2m-n-1}\widetilde F_k(y)dy \leq C(1+|x|^{2m-n-1}) \quad \text{in } \R^n.
\]
Notice that because of \eqref{est-limit-F_k} for large $k$ there holds $\widetilde F_k (x) \leq F_k (x)$ everywhere. This and \eqref{est-bound-yF_k} now implies the following
\[
\lim_{k\to\infty}\int_{\R^n \setminus B_\delta} \widetilde F_k (y)dy\to 0\quad\text{for any fixed }\delta>0.
\]
Once we have the above limit in hand and seeing $\widetilde u$ as a convolution, by standard argument, we get that 
\begin{align}\label{conv-tildeu-1}
\widetilde u_k\to \widetilde u:= |x|^{2m-n}\quad\text{in }C^0_{\rm loc}(\R^n)
\end{align}
and at the same time
\begin{align}\label{conv-tildeu-2}
\frac1C|x|^{2m-n}\leq \widetilde u_k\leq C|x|^{2m-n}\quad\text{in } \R^n \setminus B_1
\end{align}
for some $C>0$. Notice that we can write $F_k$ as 
\[
F_k=\left( \ve_k f^{2m} u_k^{1+\alpha} + f^{-c_\alpha} \right) u_k^{-\alpha}=:Q_k u_k^{-\alpha},
\]
where we denote
\[
f(x):=\frac{2}{1+|x|^2}.
\]
By the Pohozaev-type identity in Lemma \ref{Poho-Rn}, we get 
\begin{align}\label{poho-apply-1}
\int_{\R^n}(x\cdot\nabla Q_k)u_k^{1-\alpha}dx= c_\alpha \int_{\R^n} Q_k u_k^{1-\alpha}dx.
\end{align}
(Here, the multiplicative constant $\gamma_{2m,n} \ne 0$ cancels out from the both sides, thanks to Theorem \ref{thm-PDE-IE}.) Let us first compute 
\[
\nabla \big(\ve_k f^{2m} u_k^{1+\alpha} \big) = 2m \ve_k f^{2m-1} u_k^{1+\alpha} \nabla f + \frac{1+\alpha}2\ve_k f^{2m} u_k^{\alpha - 1} \nabla u_k^2 
\]
and
\[
\nabla ( f^{-c_\alpha} ) = -c_\alpha f^{-c_\alpha-1}\nabla f,
\]
leading us to
\[
x\cdot\nabla Q_k=\Big[ \left( 2m\ve_k f^{2m-1} u_k^2 -c_\alpha f^{-c_\alpha-1}u_k^{1-\alpha}\right) (x\cdot\nabla f) +\frac{1+\alpha}{2} \ve_k f^{2m} (x\cdot\nabla u_k^2) \Big]u_k^{\alpha-1} .
\]
Therefore, from \eqref{poho-apply-1} we get
\begin{align*} 
c_\alpha \int_{\R^n} \big[\ve_k f^{2m} u_k^2 + f^{-c_\alpha} u_k^{1-\alpha} \big] dx &= \int_{\R^n} \big[ 2m\ve_k f^{2m-1} u_k^2 -c_\alpha f^{-c_\alpha-1}u_k^{1-\alpha} \big] (x\cdot\nabla f) dx \\
& \quad + \frac{1+\alpha}{2} \ve_k \int_{\R^n} f^{2m} (x\cdot\nabla u_k^2) dx \\
&= \int_{\R^n} m\ve_k (1-\alpha) f^{2m-1} u_k^2 (x\cdot\nabla f) dx \\
&\quad + \int_{\R^n} \ve_k \frac{1+\alpha}2 u_k^2 (x\cdot\nabla f^{2m} ) dx \\
& \quad -c_\alpha \int_{\R^n} f^{-c_\alpha-1}u_k^{1-\alpha} (x\cdot\nabla f) dx \\
& \quad + \frac{1+\alpha}{2} \ve_k \int_{\R^n} f^{2m} (x\cdot\nabla u_k^2) dx .
\end{align*} 
By integration by parts, we note that 
\begin{align*} 
 \int_{\R^n} \big[ u_k^2 (x\cdot\nabla f^{2m}) &+ f^{2m} (x\cdot\nabla u_k^2) \big]dx \\
 &=\lim_{R \to \infty} \sum_{i=1}^n \Big[ \int_{B_R} \big[ - u_k^2 f^{2m} \big]dx + \frac{1}R \int_{\partial B_R}x_i^2 f^{2m} u_k^2 d\sigma \Big] \\
 &= \lim_{R \to \infty} \Big[ - n \int_{B_R}u_k^2 f^{2m} dx + R \int_{\partial B_R} f^{2m} u_k^2 d\sigma \Big]\\
 &=- n \int_{\R^n} f^{2m} u_k^2 dx.
\end{align*} 
Putting the above estimates together we arrive at
\begin{equation}\label{poho-apply-2}
\begin{aligned}
\ve_k\int_{\R^n} f^{2m-1}u_k^2 \Big[ m(1-\alpha)(x\cdot\nabla f)& - \Big(\frac{n(1+\alpha)}{2}+c_\alpha \Big)f\Big] dx \\
&=c_\alpha\int_{\R^n} f^{-c_\alpha-1}u_k^{1-\alpha}\left( x\cdot\nabla f +f \right) dx.
\end{aligned}
\end{equation}
Since 
\[
x\cdot\nabla f+f=f\frac{1-|x|^2}{1+|x|^2},
\]
and 
\[
m(1-\alpha)+n\frac{1+\alpha}{2}+c_\alpha =0,
\]
the identity \eqref{poho-apply-2} can be rewritten as 
\begin{align}\label{poho-apply-3} 
\ve_k m(1-\alpha)\int_{\R^n} f^{2m}u_k^2\frac{1-|x|^2}{1+|x|^2}dx= c_\alpha \int_{\R^n} f^{-c_\alpha} u_k^{1-\alpha} \frac{1-|x|^2}{1+|x|^2}dx .
\end{align} 
Our next step is to show that for large $k$, the two integrals in \eqref{poho-apply-3} are non-zero with different sign.

\noindent\textbf{Estimate of the LHS of \eqref{poho-apply-3}}.
Concerning the integral on the LHS of \eqref{poho-apply-3}, a simple calculation shows that
\begin{align*} 
 \frac{1}{M_k^2}\int_{\R^n} f^{2m} (x) & u_k^2 (x) \frac{1-|x|^2}{1+|x|^2}dx\\
&= \int_{\R^n} \big( \frac{2}{1+|x|^2}\big)^{2m} \widetilde u_k^2 (x) \frac{1-|x|^2}{1+|x|^2}dx \\
&= \int_{B_1} \big( \frac{2}{1+|x|^2}\big)^{2m}\frac{1-|x|^2}{1+|x|^2} \Big( \widetilde u_k^2 (x) - |x|^{4m-2n} \widetilde u_k^2 \big( \frac x{|x|^2} \big) \Big)dx ,
\end{align*}
here we have converted the integral on $\R^n \setminus B_1$ into $B_1$ using Kelvin's transformation. In $B_1 \setminus \{0\}$, it follows from \eqref{conv-tildeu-1} and \eqref{conv-tildeu-2} that
\[
\widetilde u_k^2 (x) - |x|^{4m-2n} \widetilde u_k^2 \big( \frac x{|x|^2} \big)
\to |x|^{4m-2n} - 1 \leq 0 \quad \text{as } k \to \infty.
\]
Notice that
\begin{align*} 
\lim_{k \to \infty} \int_{B_1}\big( \frac{2}{1+|x|^2}\big)^{2m} & \frac{1-|x|^2}{1+|x|^2} \Big( \widetilde u_k^2 (x) - |x|^{4m-2n} \widetilde u_k^2 \big( \frac x{|x|^2} \big) \Big)dx \\
&= \int_{B_1} \big( \frac{2}{1+|x|^2}\big)^{2m}\frac{1-|x|^2}{1+|x|^2}\left( |x|^{4m-2n} -1\right)dx <0.
\end{align*}
This and $\ve_k >0$ imply that the LHS of \eqref{poho-apply-3}  is  strictly negative   for large $k$.

\noindent\textbf{Estimate of the RHS of \eqref{poho-apply-3}}. Reasoning as in the previous step we should have
\begin{align*} 
\frac{1}{M_k^{1-\alpha}} & \int_{\R^n} f^{-c_\alpha} (x) u_k^{1-\alpha} (x) \frac{1-|x|^2}{1+|x|^2}dx \\
&= \int_{B_1} \big( \frac{2}{1+|x|^2}\big)^{-c_\alpha}\frac{1-|x|^2}{1+|x|^2}\Big( \widetilde u_k^{1-\alpha} (x) -|x|^{-2c_\alpha -2n} \widetilde u_k^{1-\alpha} \big( \frac x{|x|^2} \big) \Big)dx.
\end{align*}
In $B_1$, it follows from \eqref{conv-tildeu-1} that
\[
\widetilde u_k^{1-\alpha} (x) -|x|^{-2c_\alpha -2n} \widetilde u_k^{1-\alpha} \big( \frac x{|x|^2} \big) 
\to |x|^{(2m-n)(1-\alpha)} - 1 \geq 0 \quad \text{as } k \to \infty.
\]
Now observe that for $\alpha>1$
\begin{align*} 
\int_{B_1} \big( \frac{2}{1+|x|^2}\big)^{-c_\alpha}\frac{1-|x|^2}{1+|x|^2}\big(|x|^{(2m-n)(1-\alpha)} - 1 \big)dx > 0,
\end{align*}
which imply that the  the RHS of \eqref{poho-apply-3}  is strictly positive for large $k$ (for certain $\alpha > 1$, the preceding integral could be infinity). Now going back to \eqref{poho-apply-3}, we easily obtain a contradiction for $\alpha>1$. Indeed, we have two possible cases. First, if $\ve_k > 0$ for large $k$, then as $\ve_k m(1-\alpha)<0$, the LHS of \eqref{poho-apply-3} becomes strictly positive. However, as $c_\alpha \leq 0$, the RHS of \eqref{poho-apply-3} becomes non-positive. This is a contradiction. In contrary, we have $\ve_k=0$ for a sequence of $k$. However, under $\ve_k = 0$ the LHS of \eqref{poho-apply-3} vanishes but as $c_\alpha < 0$ the RHS of \eqref{poho-apply-3} becomes strictly negative. This is again a contradiction. And this completes our proof of the compactness for $\alpha>1$.

Finally we consider the case $0<\alpha\leq1$. We set 
\[
\eta_k(x):=\frac{u_k(r_k x)}{u_k(0)},\quad r_k:=u_k(0)^\frac{1+\alpha}{2m}\to0.
\] 
Then $\eta_k$ satisfies $\eta_k\geq \eta_k(0)=1$, and 
\begin{align} \label{etak-1}
\eta_k(x)=\gamma_{2m,n}\int_{\R^n}|x-y|^{2m-n} \left( \ve_k r_k^{2m} f^{2m}(r_ky)\eta_k(y) +\frac{ f^{-c_\alpha}(r_ky)}{\eta_k^\alpha(y) } \right) dy .
\end{align}
Then it follows that 
\begin{align} \label{est-1} 
\int_{\R^n}|y|^{2m-n} \left( \ve_k r_k^{2m} f^{2m}(r_ky)\eta_k(y) +\frac{ f^{-c_\alpha}(r_ky)}{\eta_k^\alpha(y) } \right) dy=\frac{\eta_k(0)}{\gamma_{2m,n}}\leq C, \end{align} and together with $\eta_k\geq1$, \begin{align} \label{est-2} \int_{\R^n}\left(1+|y|^{2m-n} \right) \frac{ f^{-c_\alpha}(r_ky)}{\eta_k^\alpha(y) } dy \leq C. \end{align} Therefore, 
\[
\eta_k(x)=\gamma_{2m,n}\ve_k r_k^{2m} \int_{B_1}|x-y|^{2m-n} f^{2m}(r_ky)\eta_k(y) dy +O(1)\quad\text{for }x\in B_1.
\] 
Integrating the above identity with respect to $x$ in $B_1$, and using that $f(r_ky)=2+o(1)$ on $B_1$, we obtain
\[ 
\int_{B_1}\eta_k(x) dx=o(1) \int_{B_1} \eta_k(y) dy +O(1),
\]
and hence
\[
\int_{B_1}\eta_kdx\leq C.
\] 
Combining the above estimates \begin{align} \label{est-3} \int_{\R^n}\left(1+|y|^{2m-n} \right)\left( \ve_k r_k^{2m} f^{2m}(r_ky)\eta_k(y) +\frac{ f^{-c_\alpha}(r_ky)}{\eta_k^\alpha(y) } \right) dy \leq C. \end{align} This yields \begin{align}\label{est-4}|\nabla \eta_k(x)|\leq C(1+|x|^{2m-n-1}),\quad \frac1C\left(1+|x|^{2m-n}\right)\leq\eta_k(x)\leq C\left(1+|x|^{2m-n}\right) .\end{align} Hence, up to a subsequence, 
\[
\eta_k\to\eta\quad\text{in }C^0_{loc}(\R^n).
\]
From Fatou's lemma, we get that 
\[
\int_{\R^n} \frac{|y|^{2m-n}}{\eta^\alpha(y)}dy<\infty,
\]
thanks to \eqref{est-4}. Since $\eta$ satisfies the second estimate in \eqref{est-4}, we necessarily have that \begin{align*} (\alpha-1)(2m-n)>n,\end{align*} a contradiction to $0<\alpha\leq1$.
\end{proof}
 
We are now in a position to prove Theorem \ref{thm-uniform}.

\begin{proof}[Proof of Theorem \ref{thm-uniform}] 
Since $\ve_k\in [0,\ve^*)$ and $0<\ve^*<1$, integrating \eqref{eq-MAIN} on $\S^n$ we get that 
\[
0 \leq \int_{\S^n}v_k \dsn \leq \frac 1{1-\ve^*} \int_{\S^n} v_k^{-\alpha}\dsn=O(1)_{k \to \infty},
\]
thanks to Lemma \ref{lem-lower}. Therefore, we arrive at
\[
\gjms (v_k) - \ve_k \Q v_k=O(1)_{k \to \infty} \quad\text{in }\S^n 
\] 
with $\|v_k\|_{L^1(\S^n)} = O(1)_{k \to \infty}$. The theorem follows from standard elliptic estimates. 
\end{proof}
 
 
\section{Moving plane arguments and proof of the main result}
 
This section is devoted to the proof of Theorem \ref{thm-main}. To obtain the symmetry of solutions, our approach is based on the method of moving planes with some new ingredients. The major difficulty is how to handle the negative exponent. As far as we know, although the method of moving planes can be effectively applied to nonlinear equations with positive exponents, see \cite{cl1991, WeiXu, clo2006, CLS} and the references therein, its applications to equations with negative exponents are very rare.

Let us recall some notation and convention often used in the method of moving planes; see Figure \ref{fig-MP} below. For $\lambda \in \R$ we set
\[
\Sigma_\lambda:=\{x\in\R^n:x_1>\lambda\},\quad T_\lambda := \partial \Sigma_\lambda .
\]
Also for any $\lambda \in \R$ we let $x^\lambda$ be the reflection of $x \in \R^n$ about the plane $T_\lambda$, namely
\[
x^\lambda :=(2\lambda-x_1, x_2,x_3,\dots, x_n).
\]
Also for any function $f$ we let $f_\lambda$ be the reflection of $f$ about the plane $T_\lambda$, namely
\[
f_\lambda (x) := f(x^\lambda) = f(2\lambda-x_1, x_2,x_3,\dots, x_n).
\]

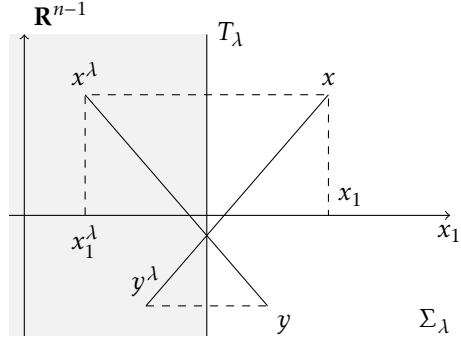
\begin{figure}[H]
\begin{tikzpicture}[scale=0.8]
\fill [gray!10!white] (-3.25,-2) rectangle (0,3);
\draw[->] (-3.25, 0) -- (4,0) node[anchor =north] {$x_1$};
\draw[-] (0,-2) -- (0,3) node[anchor =west] {$T_\lambda$};
\draw[->] (-3,-2) -- (-3,3) node[anchor =south west] {$\R^{n-1}$};
\draw[-, dashed] (-2, 0) node[anchor =north] {$x_1^\lambda$} -- (-2,2) node[anchor =south] {$x^\lambda$} -- (2,2) node[anchor =south] {$x $} -- (2,0) node[anchor =south west] {$x_1$} ;
\draw[-] (-2, 2) -- (1,-1.5) node[anchor =north west] {$y $};
\draw[-] (2, 2) -- (-1,-1.5) node[anchor =south] {$y^\lambda$};
\draw[-, dashed] (-1,-1.5) -- (1,-1.5);
\node at (3.75,-1.75) {$\Sigma_\lambda$};
\end{tikzpicture}
\caption[]{Reflection in the method of moving planes}\label{fig-MP}
\end{figure}

Throughout this section we let $u=u_\ve > 0$ be a (smooth) solution to \eqref{IEn} with $F_\ve:=F_{\ve,u}$ as in \eqref{def-F} for fixed $0<\alpha \leq (n+2m)/(2m-n)$ and fixed $0 \leq \ve<\ve^*$ with an additional assumption that $\alpha<(n+2m)/(2m-n)$ if $\ve=0$. For simplicity, we set
\[
w_{\ve, \lambda}(x) := u_\ve(x)-u_\ve(x^\lambda) \quad \text{for all } x \in \R^n.
\] 
To start moving planes, the following lemma is often required.

\begin{lemma}\label{lem-w}
There hold
\begin{equation}\label{eq-w1}
w_{\ve, \lambda}(x) = \gamma_{2m,n}\int_{\R^n} \big[ |x-y|^{2m-n}-|x^\lambda-y|^{2m-n} \big] F_\ve(y)dy
\end{equation}
and
\begin{equation}\label{eq-w2}
w_{\ve, \lambda }(x) =\gamma_{2m,n}\int_{\Sigma_\lambda} \big[ |x^\lambda-y|^{2m-n}-|x-y |^{2m-n} \big] [F_\ve(y^\lambda)-F_\ve(y )]dy
\end{equation}
for any $\lambda \in \R$.
\end{lemma}

\begin{proof} 
The first identity is obvious from the definition of $w_{\ve, \lambda}$. The second identity follows from variable changes. Indeed, one can write
\begin{align*}
u_\ve (x) &= \Big( \int_{\Sigma_\lambda} + \int_{\R^n \setminus \Sigma_\lambda}\Big) |x-y| ^{2m-n}F_\ve (y) dy\\
&= \int_{\Sigma_\lambda} |x-y| ^{2m-n}F_\ve (y) dy +\int_{\Sigma_\lambda} |x-y^\lambda |^{2m-n} F_\ve (y^\lambda ) dy \\
&= \int_{\Sigma_\lambda} |x-y|^{2m-n} F_\ve (y) dy +\int_{\Sigma_\lambda} |x^\lambda-y |^{2m-n} F_\ve (y^\lambda ) dy.
\end{align*}
Similarly, one has
\begin{align*}
u_\ve (x^\lambda ) &= \int_{\Sigma_\lambda} |x^\lambda -y|^{2m-n} F_\ve (y) dy +\int_{\Sigma_\lambda} |x -y |^{2m-n} F_\ve (y^\lambda ) dy.
\end{align*}
By putting the above identities together we arrive at the second identity.
\end{proof}

Our next step is to show that the method of moving planes can start from a very large $\lambda_0 > 0$, where $\lambda_0$ is independent of $\ve$.

\begin{lemma}\label{lem-lambda0} Let $\ve^*\in (0,1)$ be fixed. Then
there exists $\lambda_0\gg 1$ such that for every $\ve\in [0, \ve^*]$ we have
\[
w_{\ve, \lambda}(x) \geq 0 \quad \text{in } \Sigma_\lambda
\]
for $\lambda\geq\lambda_0$. 
\end{lemma}

\begin{proof} 
We start the proof by observing the existence of some constant $C>0$ such that for each $\ve\in [0,\ve^*]$ we have 
\begin{align}\label{eq-C<F<C}
\frac1C \frac{1}{1+|y|^{2m+n}}\leq F_\ve (y) \leq C\frac{1}{1+|y|^{2m+n}}\quad\text{in }\R^n;
\end{align}
see \eqref{est-infty} for a similar estimate. In the case $\ve > 0$, this simply follows from the uniform bound for $v_\ve$ with respect to $\ve\in(0, \ve^*]$ as given by Theorem \ref{thm-uniform}. In the case $\ve = 0$, the above estimate is trivial because $u(x) \approx |x|^{2m-n}$ for $|x| \gg1$. By a simple algebraic computations we have 
\begin{align*} &|x-y|^{2m-n}-|x^\lambda-y|^{2m-n}=\frac{ |x-y|^{2}-|x^\lambda-y|^{2} } { |x-y|^{2m-n}+|x^\lambda-y|^{2m-n}} \widetilde P_\lambda(x,y) ,
\end{align*} 
where the function $\widetilde P_\lambda$ is given  by 
\[
\widetilde P_{\lambda}(x,y):=\sum_{k=0}^{2m-n-1} |x-y|^{2(2m-n-1-k)}|x^\lambda-y|^{2k}.
\]
(It is clear that $\widetilde P_\lambda\equiv 1$ if $2m-n=1$.) Using \eqref{eq-w1} and 
\[
|x-y|^2-|x^\lambda -y|^2=4(x_1-\lambda)(\lambda-y_1)
\]
we can write 
\[
|x|^{2+n-2m}\frac{w_{\ve,\lambda}(x) }{x_1-\lambda}= \int_{\R^n}(\lambda-y_1) P_{\lambda}(x,y) F_\ve(y)dy=: U_\ve(x),
\] 
where 
\begin{align}\label{defP}
P_\lambda(x,y):=4\gamma_{2m,n} \frac{ |x|^{2+n-2m}}{ |x-y|^{2m-n}+|x^\lambda-y|^{2m-n}} \widetilde P_\lambda(x,y).
\end{align}
For later use, we note that for $x,y\in\Sigma_\lambda$ there holds
\begin{align} \label{estP}
P_\lambda(x,y) &\leq C|x|^{2+n-2m}\frac{|x-y|^{2(2m-n-1)} +|x^\lambda-y|^{2(2m-n-1)} }{|x-y|^{2m-n}+|x^\lambda-y|^{2m-n}} \notag\\ 
&\leq C\left\{\begin{aligned}
& \frac{|x|}{|x-y|} & & \text{for }2m-n=1\\ 
&1+|x|^{2+n-2m} |y|^{2m-n-2} & & \text{for }2m-n\geq3\end{aligned}\right. \\ 
&\leq C\left\{\begin{aligned}
& \frac{|x|}{|x-y|} && \text{for }2m-n=1\\ 
&|y|^{2m-n-2} && \text{for }2m-n\geq3\notag.\end{aligned}\right. 
\end{align} 
To conclude the lemma, it suffices to show the existence of $\lambda_0 \gg 1$ such that 
\[
U_\ve(x)>0 \quad \text{for any } x \in \Sigma_\lambda \cup T_\lambda
\]
and for every $\lambda\geq\lambda_0$. With the  help of \eqref{eq-C<F<C} we can roughly estimate
\begin{align*}
U_\ve(x) &=\int_{B_1} (\lambda-y_1) P_{\lambda}(x,y) F_\ve(y)dy 
+ \int_{\R^n \setminus B_1} (\lambda-y_1) P_{\lambda}(x,y) F_\ve(y)dy\\
&\geq\frac 1C\int_{B_1} (\lambda-y_1) P_{\lambda}(x,y) dy
 + \int_{y_1 > \lambda} (\lambda-y_1) P_{\lambda}(x,y) F_\ve(y)dy\\
&\geq\frac 1C\int_{B_1} (\lambda-y_1) P_{\lambda}(x,y) dy
 - C\int_{y_1>\lambda} \frac{P_{\lambda}(x,y)}{1+|y|^{2m+n-1}} dy\\
&=: I_{1}(x)-I_{2}(x).
\end{align*}
Here to get the term $I_2$ we have used the estimates $0 \leq y_1 - \lambda \leq y_1 \leq |y|$ in the region $\{ y \in \R^n : y_1 > \lambda\}$ and 
\[
\frac{|y|}{1+|y|^{2m+n}} \leq \frac 2{1+|y|^{2m+n-1}} \quad \text{for all } y.
\] 
Next, we estimate $I_1$ from below and $I_2$ from above. For $I_1$, we note that 
\[
P_{\lambda}(x,y)\geq\frac1C\quad\text{for } y\in B_1, \, x\in \Sigma_\lambda, \, \lambda\geq\lambda_0\gg 1.
\]
From this we deduce 
\[
I_1(x)\geq \frac\lambda C.
\]
We now estimate $I_2$. For $2m-n\geq 3$ and as 
\[
\frac{|y|^{2m-n-2}}{1+|y|^{2m+n-1}} \leq \frac 2{1+|y|^{2n+1}} \quad \text{for all } y
\] 
and $|y| \geq y_1 > \lambda $ we can estimate
\begin{align*} 
I_2(x)\leq C\int_{y_1>\lambda}\frac{|y|^{2m-n-2} dy}{1+|y|^{2m+n-1}} 
\leq C\int_{y_1>\lambda}\frac{dy}{1+|y|^{2n+1}}\leq \frac{C}{\lambda^{n+1}} \leq C. 
\end{align*} 
For $2m-n=1$, we split $\{y_1>\lambda\}$ as follows
\[
\{y_1>\lambda\}\subset A_1\cup A_2\cup A_3
\]
where 
\[
A_1:=\big\{y:\lambda<|y|\leq |x|/2\big\}, 
\quad A_2:= B_{2|x|}\setminus B_{|x|/2},
\quad A_3:= \R^n \setminus B_{2|x|}.
\]
(Although $|x| >\lambda$ as $x \in \Sigma_\lambda$, the set $A_1$ could be empty if $|x| < 2\lambda$, but it is not important.) Since $|x-y|\geq |x|/2$ on $A_1\cup A_3$ and again $|y| \geq y_1 > \lambda $, we can estimate 
\[
\int_{ A_1\cup A_3} \frac{|x|}{|x-y|} \frac{dy}{1+|y|^{2m+n-1}}\leq \frac {C}{\lambda^{2m-1}}.
\]
On the remaining set $A_2$ as $|x|/2 \leq |y| \leq 2|x|$ we easily get 
\[
\int_{A_2} \frac{|x|}{|x-y|} \frac{dy}{1+|y|^{2m+n-1}}\leq \frac{C}{|x|^{2m+n-2}}\int_{A_2}\frac{dy}{|x-y|}\leq \frac{C}{|x|^{2m-1}}\leq \frac{C}{\lambda^{2m-1}} \leq C. \] 
Putting the above estimate together, we arrive at
\[
U_\ve(x) \geq I_{1}(x)-I_{2}(x) \geq \frac\lambda C - C
\]
for some constant $C>0$. Thus, the lemma follows by letting $\lambda_0$ large enough. 
\end{proof}

In Lemma \ref{lem-lambda0}, we have compared $u_\ve (x) $ and $u_\ve (x^\lambda)$, via $w_{\ve, \lambda} (x)$, in $\Sigma_\lambda$. As there was no restriction on $\ve^* \in (0,1)$, our comparison requires large $\lambda >0$ to hold. In the next lemma, we compare $F_\ve(x)$ and $F_\ve(x^\lambda)$ in $\Sigma_\lambda$. As there will be no restriction on $\lambda>0$, our comparison now requires small $\ve>0$, and this is the place where the constant $\ve_*$ appears. Due to the form of $F_\ve$ to achieve the goal we need the compactness result established earlier; see section \ref{sec-compact}.

\begin{lemma}\label{lem-ve_*}
There exists $\ve_* \in (0, \ve^*)$ small enough such that for arbitrary $\lambda\in (0,\lambda_0]$ but fixed, the conclusion if 
\begin{equation}\label{est-w>=0}
w_{\ve,\lambda}\geq0 \quad \text{in } \Sigma_\lambda,
\end{equation} 
then
\begin{equation}\label{est-F<=F}
F_\ve(x)-F_\ve(x^\lambda)\leq 0\quad\text{in }\Sigma_\lambda
\end{equation} 
holds for each $\ve\in [0,\ve_*)$. In addition, if the inequality \eqref{est-w>=0} is strict, then so is the inequality \eqref{est-F<=F}.
\end{lemma}

\begin{proof}
Let us first be interested in the existence of $\ve_*$ and $\ve \in (0, \ve_*)$. As $|x^\lambda|<|x|$ for $\lambda>0$ and $x\in\Sigma_\lambda$, we obtain 
\begin{align*} 
F_\ve(x)-F_\ve(x^\lambda) &= \ve\big( \frac{2}{1+|x|^2}\big)^{2m} u_\ve(x)-\ve\big(\frac{2}{1+|x^\lambda|^2}\big)^{2m} u_\ve(x^\lambda) \\ 
&\quad +\big( \frac{2}{1+|x|^2}\big)^{-c_\alpha}\frac{1}{u_\ve^\alpha(x)}-\big(\frac{2}{1+|x^\lambda|^2}\big)^{-c_\alpha}\frac{1}{u_\ve^\alpha(x^\lambda)} \\ 
&\leq \ve\big( \frac{2}{1+|x|^2}\big)^{2m}(u_\ve(x)-u_\ve(x^\lambda)) \\
&\quad+\big( \frac{2}{1+|x|^2}\big)^{-c_\alpha} \Big( \frac{1}{u_\ve^\alpha(x)}- \frac{1}{u_\ve^\alpha(x^\lambda)}\Big),
\end{align*} 
where the constant $c_\alpha \leq0$ is already given in \eqref{c_alpha}. Hence, to prove \eqref{est-F<=F} in $\Sigma_\lambda$, it suffices to prove that
\begin{align}\label{est-2} 
\frac{u_\ve^\alpha(x) -u_\ve^\alpha(x^\lambda)}{u_\ve(x)-u_\ve(x^\lambda)} \frac{1}{u_\ve^\alpha(x) u_\ve^\alpha (x^\lambda)} \geq \ve \big( \frac{2}{1+|x|^2}\big)^ {(2m-n)\frac{1+\alpha}{2} }\quad\text{in }\Sigma_\lambda,
\end{align} where we have used that 
\[
2m+c_\alpha=(2m-n)\frac{1+\alpha}{2}.
\]
To this end, for some $R \gg 1$ to be specified later, we first split $\Sigma_\lambda$ into two parts as follows: 
\[
\Sigma_\lambda = \big[\Sigma_\lambda \cap B_R \big] \cup \big[\Sigma_\lambda \setminus B_R \big].
\]
In the region $\Sigma_\lambda \setminus B_R$, there exists some $\ve_1>0$ such that \eqref{est-2} holds. To see this we need to use uniform bounds with respect to $\ve>0$, see Theorem \ref{thm-uniform}, to obtain 
\[
\frac{u_\ve^\alpha(x) -u_\ve^\alpha(x^\lambda)}{u_\ve(x)-u_\ve(x^\lambda)} \geq \ve_1 \big( \frac{2}{1+|x|^2}\big)^{(2m-n)\frac{1-\alpha}{2}}
\]
and
\[
\frac{1}{u_\ve^\alpha(x) u_\ve^\alpha (x^\lambda)} \geq \ve_1 \big( \frac{2}{1+|x|^2}\big)^\alpha
\]
for some small $\ve_1 \in (0,1)$. This is mainly because when $R$ is large enough, we have $|x|\approx |x^\lambda|$ for $|x|>R$ and $\lambda\in (0,\lambda_0]$. In the region $\Sigma_\lambda \cap B_R$, by the smoothness of $u_\ve$, there exists some small $\ve_2 \in (0,1)$ such that 
\begin{align}\label{est-5.5} 
\frac{u_\ve^\alpha(x) -u_\ve^\alpha(x^\lambda)}{u_\ve(x)-u_\ve(x^\lambda)} \frac{1}{u_\ve^\alpha(x) u_\ve^\alpha (x^\lambda)} \geq \ve_2 \big( \frac{2}{1+|x|^2}\big)^{(2m-n)\frac{1+\alpha}{2}}
\end{align}
for any $x\in B_R$. Hence, combining \eqref{est-2} and \eqref{est-5.5} yields the desired estimate \eqref{est-F<=F} with 
\[
\ve_*= \frac 12 \min\{\ve_1 ,\ve_2\}.
\]
Now we consider the remaining case $\ve = 0$. However, this case is trivial because
\begin{align*} 
F_0 (x)-F_0 (x^\lambda) &= \big( \frac{2}{1+|x|^2}\big)^{-c_\alpha} \Big( \frac{1}{u_0^\alpha(x)}- \frac{1}{u_0^\alpha(x^\lambda)}\Big) \leq 0
\end{align*} 
whenever $w_{0,\lambda} (x) = u_0 (x)-u_0 (x^\lambda) \geq 0$. Finally, from the above calculation, it is clear that if the inequality \eqref{est-w>=0} is strict, then the inequality \eqref{est-F<=F} is also strict. Hence, the lemma is proved.
\end{proof}
 
Thanks to Lemma \ref{lem-lambda0}, for each $\ve>0$ we can set 
\[
\overline \lambda_\ve:=\inf \big\{ \lambda>0: w_{\ve, \mu}\geq0 \; \text{ in } \; \Sigma_\mu \; \text{ for every } \; \mu\geq\lambda \big\}.
\]
Then, still by Lemma \ref{lem-lambda0}, we necessarily have 
\[
0 \leq \overline \lambda_\ve\leq\lambda_0.
\]
Our goal is to  show that $\overline \lambda_\ve = 0$. This can be done through two steps. First we show that if $\overline \lambda_\ve > 0$, then we must have $w_{\ve, \overline \lambda_\ve} \equiv 0$ in $\Sigma_{\overline \lambda_\ve}$; see Lemma \ref{lem-barlambda>0}. Finally, we show that $\overline \lambda_\ve = 0$; see Lemma \ref{lem-barlambda=0}. 

Our next lemma is of importance to achieve the first step as it allows us to move $\lambda$ to the left.

\begin{lemma}\label{lem-1} 
Let $\ve\in [0,\ve_*)$ and $\bar\lambda\in(0,\lambda_0]$ be such that 
\[
0\not\equiv w_{\ve,\bar\lambda}\geq0 \quad \text{in } \Sigma_{\bar\lambda}.
\]
Then, there exist $R \gg 1$ and $\delta>0$ small, both may depend on $w_{\ve,\bar\lambda}$, such that for every $\lambda\in ( \bar\lambda-\delta ,\bar\lambda)$ we have 
\[
w_{\ve,\lambda} >0\quad\text{in } \Sigma_{\lambda}\setminus B_R.
\]
\end{lemma}

\begin{proof} 
Using the representation \eqref{eq-w2}, and as in the first part of the proof of Lemma \ref{lem-lambda0}, we have
\begin{align}\label{eq-5.6}
w_{\ve, \lambda }(x)\frac{|x|^{2+n-2m}}{x_1-\lambda}&= \int_{\Sigma_\lambda} (y_1-\lambda)P_\lambda(x,y) [F_\ve(y^\lambda)-F_\ve(y )]dy,
\end{align} 
where $P_\lambda$ is given by \eqref{defP}. In view of \eqref{eq-5.6}, it suffices to show that its RHS is positive in $\Sigma_\lambda \setminus B_R$ for suitable $R>0$. For convenience, we recall the following formula for $P_\lambda$
\begin{align*}
P_\lambda(x,y) =4 \gamma_{2m,n}\frac{ |x|^{2+n-2m} }{ |x-y|^{2m-n}+|x^\lambda-y|^{2m-n}} \sum_{k=0}^{2m-n-1} |x-y|^{2(2m-n-1-k)}|x^\lambda-y|^{2k}.
\end{align*}
Hence, there exists some $\theta>0$ such that for every $R_1>0$ fixed
\begin{align}\label{expan-P} 
P_\lambda(x,y) \rightrightarrows \theta
\quad \text{ uniformly in } y\in B_{R_1}
\end{align}
as $|x|\to\infty$. This is because $|x| \approx |x-y| \approx |x^\lambda - y|$ for large $|x|$. From \eqref{est-F<=F} we know that
\[
0\not\equiv F_\ve(y^{\bar\lambda})-F_\ve(y)\geq0\quad\text{for }y\in \Sigma_{\bar\lambda},
\]
which implies
\[
\int_{\Sigma_{\bar\lambda}}(y_1-\bar\lambda) [F_\ve(y^{\bar\lambda})-F_\ve(y )]dy\geq 2c_0>0,
\]
for some small constant $c_0>0$. Thus, by the dominated convergence theorem, we can find some $\delta>0$ such that 
\begin{align}\label{c_0} 
\int_{\Sigma_\lambda}(y_1-\lambda) [F_\ve(y^\lambda)-F_\ve(y )]dy\geq c_0>0 ,
\end{align}
for every $|\lambda-\bar\lambda|<\delta$. To obtain the positivity of the right hand side of \eqref{eq-5.6}, we split the integral $\int_{\Sigma_\lambda}$ into two parts as follows 
\[
\int_{\Sigma_\lambda} = \int_{\Sigma_\lambda\setminus B_{R_2}} + \int_{\Sigma_\lambda \cap B_{R_2}}
\]
for some $R_2>0$ to be determined later and estimate these integrals term by term; see the two estimates \eqref{estimate-outside} and \eqref{estimate-inside} below. Our aim is to show that the integral $\int_{\Sigma_\lambda\setminus B_{R_2}}$ is negligible.

We assume for a moment that such a constant $R_2$ exists. We now estimate the integral $\int_{\Sigma_\lambda\setminus B_{R_2}} $. First we  choose a $R_0 \gg 1$ in such a way that $|y_1-\lambda| < 2|y|$ for all $|y| \geq R_0$. Then we find some $R_1 \gg R_0$ such that in $\Sigma_\lambda \setminus B_{R_1}$ we have
\begin{equation}\label{estimate-outside-noP}
\int_{\Sigma_\lambda\setminus B_{R_1}} \frac{ dy}{1+|y|^{2m+n -1}} 
\leq \frac {\theta c_0}{16C}
\end{equation} 
and
\[
F_\ve(y)+F_\ve(y^\lambda)\leq\frac{C}{1+|y|^{2m+n}}
\]
for some $C>0$ because $|y| \approx |y^\lambda|$. By the estimate \eqref{estP} for $P_\lambda$, we now claim that there are some $R_3 \gg 1$ and $R_2 \gg R_1$ such that 
\begin{equation}\label{estimate-outside}
\int_{\Sigma_\lambda\setminus B_{R_2}}(y_1-\lambda) P_\lambda(x,y)[F_\ve(y^\lambda)+F_\ve(y )]dy\leq \frac {\theta c_0}4
\end{equation} 
for $|x| \geq R_3$. To see this, for clarity, we consider the two cases $2m-n=1$ and $2m-n \geq 3$ separately. 

\noindent\textbf{Case 1}. Suppose $2m-n=1$. In this case our estimate for $P_\lambda$ becomes $P_\lambda (x,y) \leq C |x|/|x-y|$. Consequently, there holds
\[
\int_{\Sigma_\lambda\setminus B_{R_2}}(y_1-\lambda) P_\lambda(x,y)[F_\ve(y^\lambda)+F_\ve(y )]dy\leq C\int_{\Sigma_\lambda\setminus B_{R_2}} \frac{|x|}{|x-y|} \frac{ |y|}{1+|y|^{2m+n}} dy.
\]
For $|x| \geq R_3 \gg 2 R_2$ to be determined later, we now split $\int_{\Sigma_\lambda\setminus B_{R_2}}$ as follows
\[
\int_{\Sigma_\lambda\setminus B_{R_2}}
= \int_{[\Sigma_\lambda\setminus B_{R_2}] \cap [B_{|x|/2} \cup ( \R^n \setminus B_{2|x|} )]}
+ \int_{[\Sigma_\lambda\setminus B_{R_2}] \setminus [B_{|x|/2} \cup ( \R^n \setminus B_{2|x|} )]}.
\]
Thanks to \eqref{estimate-outside-noP}, we get
\[
C \int_{[\Sigma_\lambda\setminus B_{R_2}] \cap [B_{|x|/2} \cup ( \R^n \setminus B_{2|x|} )]}
 \frac{|x|}{|x-y|} \frac{ |y|}{1+|y|^{2m+n}} dy
< \frac {\theta c_0}{8} .
\]
For the remaining integral on $[\Sigma_\lambda\setminus B_{R_2}] \setminus [B_{|x|/2} \cup ( \R^n \setminus B_{2|x|} )]$ which is a subset of $B_{2|x|} \setminus B_{|x|/2}$ because $|x| \geq 2R_2$, we estimate as follows
\[
C\int_{[\Sigma_\lambda\setminus B_{R_2}] \setminus [B_{|x|/2} \cup ( \R^n \setminus B_{2|x|} )]}
\leq \frac {C |x|^2}{1+|x|^{2m+n}} \int_{B_{2|x|} \setminus B_{|x|/2}} \frac{dy}{|x-y|}.
\]
Since the last integral is of order $|x|^n$ and $m \geq 2$ we can find some $R_3 \gg 1$ such that
\[
\frac {C |x|^2}{1+|x|^{2m+n}} \int_{B_{2|x|} \setminus B_{|x|/2}} \frac{dy}{|x-y|} \leq \frac { \theta c_0}{8}
\]
for all $x \in \Sigma_\lambda \cap B_{R_3}$. Combining the two estimates above gives \eqref{estimate-outside}. This completes the first case.

\medskip
\noindent\textbf{Case 2}. Suppose $2m-n \geq 3$. This case is easy to handle. Recall that our estimate for $P_\lambda$ becomes $P_\lambda (x,y) \leq C |y|^{2m-n-2}$. Consequently, there holds
\[
\int_{\Sigma_\lambda\setminus B_{R_2}}(y_1-\lambda) P_\lambda(x,y)[F_\ve(y^\lambda)+F_\ve(y )]dy\leq C\int_{\Sigma_\lambda\setminus B_{R_2}} \frac{ |y|^{2m-n-1}}{1+|y|^{2m+n}} dy.
\]
Seeing \eqref{estimate-outside-noP} or as in the proof of Lemma \ref{lem-lambda0}, we easily obtain the desired estimate.

Hence, up to this point, we have already shown that there are some $R_2 \gg 1$ and $R_3 \gg 1$ such that the estimate \eqref{estimate-outside} holds for $|x| \geq R_3$. Now we estimate the integral $\int_{\Sigma_\lambda \cap B_{R_2}}$. Keep using the constant $R_2$. By the uniform convergence in \eqref{expan-P}, we can choose $R_4\gg R_2$ such that 
\[
P_\lambda(x,y)\geq \frac12\theta\quad\text{for }|x|\geq R_4 \text{ and } |y|\leq R_2. 
\]
This and \eqref{c_0} imply that 
\begin{equation}\label{estimate-inside}
\int_{\Sigma_\lambda \cap B_{R_2}}(y_1-\lambda) P_\lambda(x,y)[F_\ve(y^\lambda)-F_\ve(y )]dy\geq \frac {\theta c_0}2 
\end{equation} 
for $|x| \geq R_4$. We conclude the lemma by combing the two estimates \eqref{estimate-outside} and \eqref{estimate-inside} and choosing $R = \max\{R_3, R_4\}$.
\end{proof}

We are now in a position to complete the first step, namely, to show that $\overline \lambda_\ve = 0$. To this purpose, we must rule out the case $\overline \lambda_\ve > 0$ and this is the content of the next two lemmas. First, we characterize the function $w_{\ve,\overline \lambda_\ve}$ in case $\overline \lambda_\ve > 0$.

\begin{lemma}\label{lem-barlambda>0}
If $\overline \lambda_\ve>0$ for some $\ve\in [0,\ve_*)$, then $w_{\ve,\overline \lambda_\ve}\equiv0$ in $\Sigma_{\overline \lambda_\ve}$. In other words, the function $u_\ve$ is symmetric with respect to the hyperplane $\{ x \in \R^n :x_1 = \overline \lambda_\ve\}$.
\end{lemma}

\begin{proof}
Let $\overline \lambda_\ve>0$ for some $\ve\in [0,\ve_*)$ and assume by contradiction that $w_{\ve,\overline \lambda_\ve}\not\equiv0$ in $\Sigma_{\overline \lambda_\ve}$. 
This and the definition of $\overline \lambda_\ve$ imply that
\[
0 \not\equiv w_{\ve,\overline \lambda_\ve} \geq 0 \quad \text{in } \Sigma_{\overline \lambda_\ve}.
\]
By Lemma \ref{lem-1}, there exist $R \gg1$ and $\delta>0$ small enough such that
\[
w_{\ve,\lambda} >0\quad\text{in } \Sigma_{\lambda}\setminus B_R \quad\text{for every }\lambda\in ( \overline \lambda_\ve-\delta ,\overline \lambda_\ve). 
\]
Then there exists a sequence $\mu_k\nearrow\overline\lambda_\ve$ such that $w_{\ve, \mu_k}$ is negative somewhere in $\Sigma_{\mu_k}$. Since outside $B_R$, the function $w_{\ve, \mu_k}$ is strictly positive, for each $k$ there is some $x_k \in \Sigma_{\mu_k} \cap \overline{B_R}$ such that
\[
w_{\ve,\mu_k}(x_k) = \min_{\Sigma_{\mu_k}}w_{\ve,\mu_k} <0.
\]
In particular, there holds
\[
\frac{w_{\ve, \mu_k}(x_k)}{(x_k)_1-\mu_k}<0.
\] 
Obviously, the sequence $(x_k)$ is bounded as $x_k \in \overline{B_R}$. Also note that $\Sigma_{\overline \lambda_\ve} \subset \Sigma_{\mu_k}$ and $\Sigma_{\mu_k} \searrow \Sigma_{\overline \lambda_\ve}$ as $k \nearrow +\infty$. Therefore, up to a subsequence, we have
\[
\Sigma_{\overline \lambda_\ve} \cup T_{\overline \lambda_\ve} \ni x_\infty:=\lim_{k\to\infty} x_k. 
\]
In particular, by passing to the limit as $k \to \infty$, there holds $w_{\ve, \overline \lambda_\ve}(x_\infty) \leq 0$. This and \eqref{eq-5.6} implies that
\begin{align*}
0 \geq w_{\ve, \overline \lambda_\ve }(x_\infty)\frac{|x_\infty|^{2+n-2m}}{(x_\infty)_1-\overline \lambda_\ve}&= \int_{\Sigma_{\overline \lambda_\ve}} (y_1-\overline \lambda_\ve)P_{\overline \lambda_\ve} (x_\infty,y) [F_\ve(y^{\overline \lambda_\ve})-F_\ve(y )]dy \geq 0,
\end{align*} 
thanks to $|x_\infty| > 0$ and $F_\ve(y^{\overline \lambda_\ve}) \geq F_\ve(y )$ in $\Sigma_{\overline \lambda_\ve}$ by Lemma \ref{lem-ve_*}.
Thus, we must have
\[
F_\ve(y^{\bar \lambda\ve})-F_\ve(y ) = 0\quad\text{for any }y\in \Sigma_{\overline \lambda_\ve},
\]
which, by \eqref{eq-5.6}, now yields $w_{\ve,\overline \lambda_\ve}\equiv 0$ in $\Sigma_{\overline \lambda_\ve}$. However, this is a contradiction.  Once we have $w_{\ve,\overline \lambda_\ve}\equiv 0$ in $\Sigma_{\overline \lambda_\ve}$, the symmetry of $u_\ve$ follows from the definition of $w_{\ve,\overline \lambda_\ve}$. The proof is complete.
\end{proof}

From the characterization of $w_{\ve,\overline \lambda_\ve}$ in the case $\overline \lambda_\ve>0$ and the role of the size of $\ve$ and $\alpha$, we are able to show that in fact the case $\overline \lambda_\ve > 0$ cannot happen.

\begin{lemma}\label{lem-barlambda=0}
Let $\ve \in [0, \ve_*)$. There holds $\overline \lambda_\ve=0$. In particular, the function $u_\ve$ is symmetric with respect to the hyperplane $\{ x \in \R^n :x_1 = 0\}$.
\end{lemma}

\begin{proof} 
By way of contradiction, assume that $\overline \lambda_\ve>0$. In view of Lemma \ref{lem-barlambda>0}, we must have
\[
0=w_{\ve,\overline \lambda_\ve} (x) = u_\ve(x)-u_\ve(x^{\overline \lambda_\ve})
\]
in $\Sigma_{\overline \lambda_\ve}$. This and \eqref{eq-w2} tell us that
\begin{align*}
\int_{\Sigma_{\overline \lambda_\ve}} \big[|x^{\overline \lambda_\ve} -y|^{2m-n}-|x-y |^{2m-n} \big] [F_\ve(y^{\overline \lambda_\ve})-F_\ve(y )]dy=0
\end{align*}
for any $x \in \Sigma_{\overline \lambda_\ve}$, thanks to $ \gamma_{2m,n} \ne 0$, see Theorem \ref{thm-PDE-IE}. But this cannot happen because $|x -y| \leq |x^{\overline \lambda_\ve}-y | $ for any $x, y \in \Sigma_{\overline \lambda_\ve}$ and
\begin{align*} 
F_\ve(x)-F_\ve(x^{\overline \lambda_\ve}) &= \ve \Big[ \big( \frac{2}{1+|x|^2}\big)^{2m} - \big(\frac{2}{1+|x^{\overline \lambda_\ve}|^2}\big)^{2m} \Big] u_\ve(x) \\ 
&\quad +\Big[ \big( \frac{2}{1+|x|^2}\big)^{-c_\alpha} -\big(\frac{2}{1+|x^{\overline \lambda_\ve}|^2}\big)^{-c_\alpha} \Big] \frac{1}{u_\ve^\alpha(x)} \\ 
&<0,
\end{align*} 
everywhere in $\Sigma_{\overline \lambda_\ve}$, thanks to the estimates $u_\ve > 0$, $-c_\alpha \geq 0$, and $|x| \leq |x^{\overline \lambda_\ve}|$ in $\Sigma_{\overline \lambda_\ve}$. (Here we also use the fact that if $\ve=0$, then $\alpha<(n+2m)/(2m-n)$ in order to guarantee $-c_\alpha>0$.) Thus, we must have $\overline \lambda_\ve=0$. In particular, we have from the definition of $\overline \lambda_\ve$ the following
\[
u_\ve (x_1, x_2,...,x_n) \geq u_\ve (-x_1, x_2,...,x_n).
\]
We now apply the method of moving planes in the opposite direction, namely $\lambda < 0$, to get
\[
u_\ve (x_1, x_2,...,x_n) \leq u_\ve (-x_1, x_2,...,x_n).
\]
Hence
\[
u_\ve (x_1, x_2,...,x_n) = u_\ve (-x_1, x_2,...,x_n).
\]
This establishes the symmetry of $u_\ve$ with respect to the hyperplane $\{ x \in \R^n :x_1 = 0\}$. The proof is now complete.
\end{proof}

We now have a quick note. In the proof of Lemma \ref{lem-barlambda=0} above, we crucially use the hypothesis that either $\ve > 0$ and $\alpha \leq (n+2m)/(n-2m)$ or $\ve =0$ and $\alpha < (n+2m)/(n-2m)$. For the latter case, if $\ve =0$ and $\alpha = (n+2m)/(n-2m)$, then we cannot claim that $\overline \lambda_0=0$. Therefore, we could only claim that $u_0$ is radially symmetric with respect to some point not necessarily the origin. This leads to explicit form of non-trivial solutions to \eqref{eq-MAIN}$_0$ in the conformaly invariant case.

As a consequence of Lemma \ref{lem-barlambda=0} above, we obtain a Liouville type result for positive, smooth solution to \eqref{eq-MAIN}$_\ve$ for small $\ve>0$, hence proving Theorem \ref{thm-main}.

\begin{lemma}\label{lem-Liouville}
Any positive, smooth solution $v_\ve$ to \eqref{eq-MAIN}$_\ve$ for small $\ve$ must be constant.
\end{lemma}

\begin{proof} 
Let $\ve \in [0,\ve_*)$ be arbitrary. From Lemma \ref{lem-barlambda=0} we know that the corresponding solution $u_\ve$ is symmetric with respect to the hyperplane $\{ x \in \R^n :x_1 = 0\}$. This together with the relation
\[
u_\ve (x) = \Big( \frac {1+|x|^2} 2 \Big)^\frac{2m-n}{2} \big( v_\ve \circ \pi_N^{-1} \big) (x)
\]
tells us that $v_\ve$ depends only on the last coordinate $x_{n+1}$. However, as the $x_{n+1}$-axis is freely chosen, we conclude that $v_\ve$ must be constant. This completes the proof.
\end{proof}

Before closing this section, we have a remark. To obtain the symmetry of solutions to \eqref{eq-MAIN}$_\ve$ for small $\ve$, our approach is based on the method of moving planes in the integral form. A natural question is weather or not one can use the method of moving spheres; see \cite{LZ95, Li}. Due to the presence of the weight $2/(1+|x|^2)$ in \eqref{def-F}, it is natural to ask whether or not the method of moving spheres can still be used. Toward a possible answer to this question, we refer the reader to the work \cite{JLX08}.

 
\section{Application to the sharp Sobolev inequality}
\label{sec-SS}

This section is devoted to a proof of Theorem \ref{thm-SS} which concerns a sharp (critical or subcritical) Sobolev inequality. Let $\ve \in (0, 1)$ and inspired by \eqref{eq-EP} consider the following variational problem
\begin{equation}\label{eq-SS-O}
\mathcal S_\ve = \inf_{0<\phi \in H^m (\S^n) } \Big( \int_{\S^n} \phi^{1-\alpha} \dsn \Big)^{\frac 2{\alpha-1}}
\int_{\S^n} \big[ \phi \gjms (\phi) - \ve \Q \phi^2 \big] \dsn
\end{equation}
with $m=(n+1)/2$ and $\alpha \in(0,1) \cup (1, 2n+1]$. We note that although the constant $(n-2m)/2$ becomes $-1/2$ in the present case, we intent to keep it in various calculation below for convenience. Similar convention also applies for $\gjms$ instead of $\mathbf P_n^{n+1}$, etc. Now as 
\[
\gjms (1) - \ve \Q =(1- \ve )\Q \ne 0
\]
by testing \eqref{eq-SS-O} with constant functions we conclude from \eqref{eq-SS-O} that
\[
\mathcal S_\ve \leq (1- \ve )\Q |\S^n|^\frac{\alpha+1}{\alpha -1}<0,
\] 
however, $\mathcal S_\ve $ could be $-\infty$. Next we show that $\mathcal S_\ve$ is finite and is achieved by some smooth positive function. 
 
\begin{lemma}\label{lem-SSexistence}
Assume that $m=(n+1)/2$ and $\alpha \in(0,1) \cup (1, 2n+1]$. Then, the constant $\mathcal S_\ve$ in \eqref{eq-SS-O} is finite and there exists some $v_\ve \in C^\infty (\S^n)$ such that $v_\ve>0$ and
\[
\Big( \int_{\S^n} v_\ve^{1-\alpha} \dsn \Big)^{\frac 2{\alpha-1}} \int_{\S^n} \big[ v_\ve \gjms (v_\ve) - \ve \Q v_\ve^2 \big] \dsn = \mathcal S_\ve.
\]
In particular, $v_\ve$ solves
\[
\gjms (v_\ve) - \ve \Q v_\ve = S_\ve v_\ve^{-\alpha}
\]
in $\S^n$ with
\[
S_\ve = \frac{\mathcal S_\ve}{\|v_\ve^{-1}\|_{L^{\alpha-1} (\S^n)}^{\alpha+1}}. 
\] 
\end{lemma}

\begin{proof}
Let $(v_k)_k$ be a positive, smooth minimizing sequence in $H^{2m} (\S^n)$, that is \[
\Big( \int_{\S^n} v_k^{1-\alpha} \dsn \Big)^{\frac 2{\alpha-1}} \int_{\S^n} \big[ v_k \gjms (v_k) - \ve \Q v_k^2 \big] \dsn \searrow \mathcal S_\ve
\]
as $k \to \infty$. By the scaling invariant we can assume $\max_{\S^n} v_k =1$ which then yields 
\[
\|v_k \|_{L^2 (\S^n)}^2 \leq |\S^n|. 
\]
As $\gjms$ is a monic polynomial of $-\DSn$, the coefficient of the highest degree is equal to $1$, it is easy to get that
\[
\int_{\S^n} v_k \gjms (v_k) \dsn 
\geq c_1 \|v_k\|_{H^m(\S^n)}^2 - c_2 \|v_k\|_{L^2(\S^n)}^2 \geq c_1 \|v_k\|_{H^m(\S^n)}^2 - c_2 |\S^n|
\]
for some $c_1>0$ and $c_2> 0$. Note that $\mathcal S_\ve<0$ and $Q^{2m}_n>0$ would imply 
\[
\int_{\S^n}v_k \gjms v_k\dsn<0. 
\]
Therefore, the previous estimate leads to 
\[
 c_1 \|v_k\|_{H^m(\S^n)}^2 \leq c_2 |\S^n|,
\]
giving the boundedness of the sequence $(v_k)$ in $H^m(\S^n)$. Hence, after passing to a subsequence if necessary, there exists some $v_\ve \in H^m(\S^n)$ such that
\[
v_k \to v_\ve \geq 0 \text{ uniformly in } C (\S^n)
\]
by Morrey's inequality and the Arzel\`a--Ascoli lemma, and
\[
v_k \rightharpoonup v_\ve \text{ weakly in } H^m (\S^n).
\]
In particular, there holds $\max_{\S^n} v_\ve = 1$. As $v_\ve \geq 0$, there are two possibilities. First, let us assume that $v_\ve$ vanishes somewhere on $\S^n$. By assuming this we shall obtain a contradiction, therefore we must have $v_\ve>0$. Indeed, as $n=2m-1$, we can make use of \cite[Corollary 3.1]{Hang} to conclude that 
\[
\int_{\S^n} v_\ve \gjms (v_\ve) \dsn \geq 0.
\]
This together with $\ve \Q < 0$ and $\int_{\S^n}v_\ve^2 \dsn>0$ help us to get
\begin{align*}
0 & <\int_{\S^n} \big[ v_\ve \gjms (v_\ve) - \ve \Q v_\ve^2 \big] \dsn \\
&\leq \liminf_{k \nearrow +\infty} \int_{\S^n} \big[ v_k \gjms (v_k) - \ve \Q v_k^2 \big] \dsn.
\end{align*}
This is a contradiction to $\mathcal S_\ve<0$. Thus, $v_\ve > 0$ everywhere. Then, this allows us to gain
\[
v_k^{-1} \to v_\ve^{-1} \quad \text{uniformly in } C (\S^n)
\]
and consequently
\[
\int_{\S^n} v_k^{1-\alpha} \dsn \to \int_{\S^n} v_\ve^{1-\alpha} \dsn .
\]
Putting these facts together, we obtain
\begin{equation}\label{eq-S<quotient<S}
\begin{aligned}
\mathcal S_\ve & \leq \Big( \int_{\S^n} v_\ve^{1-\alpha} \dsn \Big)^{\frac 2{\alpha-1}} \int_{\S^n} \big[ v_\ve \gjms (v_\ve) - \ve \Q v_\ve^2 \big] \dsn \\
& \leq \liminf_{k \nearrow +\infty} \Big[ \Big( \int_{\S^n} v_k^{1-\alpha} \dsn \Big)^{\frac 2{\alpha-1}} \int_{\S^n} \big[ v_k \gjms (v_k) - \ve \Q v_k^2 \big] \dsn \Big]\\
&=\mathcal S_\ve.
\end{aligned}
\end{equation}
Hence, on one hand implies that $\mathcal S_\ve$ must be finite, on the other hand, yields that $v_\ve$ is a minimizer for \eqref{eq-SS-O}. Rest of the proof follows immediately. 
\end{proof}

Now  we are in a position to give a proof of  Theorem \ref{thm-SS}. 

\begin{proof}[Proof of Theorem \ref{thm-SS}]
Let $\ve>0$ and $\alpha \in(0,1) \cup (1, 2n+1]$. By Lemma \ref{lem-SSexistence}, there is some positive, smooth function $v_\ve$ satisfying
\[
\int_{\S^n} v_\ve^{1-\alpha} \dsn = 1
\] 
and
\[
\int_{\S^n} \big[ v_\ve \gjms (v_\ve) - \ve \Q v_\ve^2 \big] \dsn = \mathcal S_\ve.
\]
Then, up to a constant multiple, $v_\ve$ solves \eqref{eq-MAIN}$_\ve$ in $\S^n$. Therefore, for small $\ve>0$, it follows from Theorem \ref{thm-main} that $v_\ve$ is constant. Keep in mind that $\alpha \ne 1$. Hence, on one hand, as $((n-2m)/2) Q_n^{2m} = \gjms (1)$, we can compute to get
\[
\mathcal S_\ve = ( 1- \ve) \Q |\S^n|^\frac{\alpha+1}{\alpha-1},
\]
on the other hand, by the definition of $\mathcal S_\ve$ we get
\begin{align*}
\Big( \int_{\S^n} \phi^{1-\alpha} \dsn \Big)^{\frac 2{\alpha-1}} 
& \int_{\S^n} \big[ \phi \gjms (\phi) - \ve \Q \phi^2 \big] \dsn \\
& \geq ( 1- \ve) \Q | \S^n|^\frac{\alpha+1}{\alpha-1}
\end{align*}
for any $\phi \in H^m (\S^n)$ with $\phi>0$. Now letting $\ve \searrow 0$ we obtain
\[
\Big( \int_{\S^n} \phi^{1-\alpha} \dsn \Big)^{\frac 2{\alpha-1}} 
\int_{\S^n} \phi \gjms (\phi) \dsn
\geq \Q | \S^n|^\frac{\alpha+1}{\alpha-1}.
\]
Recall that 
\[
\frac{n-2m}2 Q_n^{2m} = \gjms (1) = \frac{ \Gamma ( n/2 + m) }{ \Gamma (n/2 - m )}.
\] 
This completes the proof of Theorem \ref{thm-SS}.
\end{proof}

Before closing this section, let us revisit the last comment in Remark \ref{rmk}. For convenience, let us relabel \eqref{eq-sSobolev} as follows
\begin{subequations}\label{eq-Sobolev}
\begin{align}
 \Big( \int_{\S^n} \phi^{1-\alpha} \dsn \Big)^{\frac {2}{\alpha -1}} \int_{\S^n} \phi \gjms (\phi) \dsn
\geq \frac{\Gamma (n/2 + m)}{\Gamma (n/2 - m )} | \S^n|^\frac{\alpha + 1}{\alpha - 1}.
\tag*{ \eqref{eq-Sobolev}$_\alpha$}
\end{align}
\end{subequations}
We shall establish the following, which  has its own interest.

\begin{proposition}\label{prop-order}
There holds
\begin{center}
\eqref{eq-Sobolev}$_{2n+1}$ $\longrightarrow$ \eqref{eq-Sobolev}$_\beta$ with $\beta \in (1,2n+1)$ $\longrightarrow$ \eqref{eq-sSobolev-limit} $\longrightarrow$ \eqref{eq-Sobolev}$_\alpha$ with $\alpha \in (0,1)$,
\end{center}
where the notation $A \longrightarrow B$ means we can obtain $B$ from $A$.
\end{proposition}

Before proving Proposition \ref{prop-order} we observe that, as $\Gamma(n/2+m)/\Gamma(n/2-m) < 0$, if $n=2m-1$, our related inequalities are only meaningful if 
\begin{equation}\label{eq-convention}
\int_{\S^n}\phi \gjms (\phi) \dsn<0.
\end{equation}
Therefore, from now on we always assume the above inequality. Besides, one can simplify the computation below by normalizing the measure on $\S^n$ in such a way that $|\S^n|=1$. However, we intend to keep it for clarity.

\begin{proof}[Proof of Proposition \ref{prop-order}]
Let us establish all $\longrightarrow$ each by each.

\noindent
\textbf{Proof of \eqref{eq-Sobolev}$_{2n+1}$ $\longrightarrow$ \eqref{eq-Sobolev}$_\beta$ with $\beta \in (1,2n+1)$}. Let $\beta \in (1,2n+1)$ be arbitrary but fixed. We wish to derive \eqref{eq-Sobolev}$_\beta$ from \eqref{eq-Sobolev}$_{2n+1}$. Thanks to $0<\beta-1<2n$, we can apply H\"older's inequality in the following way
\[
\int_{\S^n} \big( \phi^{-1} \big)^{\beta - 1} \dsn \leq |\S| ^\frac{2n+1-\beta}{2n} \Big(\int_{\S^n} \big( \phi^{-1} \big)^{2n} \dsn \Big)^\frac{\beta -1 }{2n}
\]
to get
\[
\Big( \int_{\S^n} \phi^{1-\beta} \dsn \Big)^{\frac {2}{\beta -1}} 
\leq |\S| ^\frac{2n+1-\beta}{n(\beta-1)} \Big( \int_{\S^n} \phi^{-2n} \dsn \Big)^{\frac 1{n}}.
\]
From this and \eqref{eq-convention} one immediately obtains
\begin{align*}
\Big( \int_{\S^n} \phi^{1-\beta} &\dsn \Big)^{\frac {2}{\beta -1}} 
\int_{\S^n}\phi \gjms (\phi) \dsn\\
& \geq |\S^n| ^\frac{2n+1-\beta}{n(\beta-1)} \Big( \int_{\S^n} \phi^{-2n} \dsn \Big)^{\frac 1{n}}\int_{\S^n}\phi \gjms (\phi) \dsn.
\end{align*}
With help of \eqref{eq-Sobolev}$_{2n+1}$ and the identity
\[
\frac{2n+1-\beta}{n(\beta-1)}+ \frac{2m}n = \frac{\beta+1}{\beta-1}
\]
we obtain \eqref{eq-Sobolev}$_\beta$ as claimed. (Keep in mind that $2m=n+1$.) This shows the first $\longrightarrow$ from the left.

\noindent
\textbf{Proof of \eqref{eq-Sobolev}$_\beta$ with $\beta \in (1,2n+1)$ $\longrightarrow$ \eqref{eq-sSobolev-limit}}. We now consider arbitrary but fixed $\beta \in (1,2n+1)$ and we wish to derive \eqref{eq-sSobolev-limit} from \eqref{eq-Sobolev}$_\beta$. By Jensen's integral inequality of the form 
\begin{equation}\label{eq-Jensen}
\frac 1{|\S^n|}\int_{\S^n}\log \psi \, \dsn \leq
\log \Big( \frac 1{|\S^n|}\int_{\S^n}\psi \, \dsn \Big)
\end{equation} 
we know by choosing $\psi = \phi^{-\gamma}$ that
\begin{equation}\label{eq-sSobolev-limit-2}
\exp \Big(- \frac 2{|\S^n|}\int_{\S^n}\log  \phi \, \dsn \Big) 
\leq \Big( \frac 1{|\S^n|}\int_{\S^n}  \phi^{-\gamma} \, \dsn\Big)^{2/\gamma}
\end{equation} 
for any $\gamma \in \R$. In \eqref{eq-sSobolev-limit-2} we choose $\gamma = \beta - 1$ and together with \eqref{eq-convention} we eventually get
\begin{align*}
\exp \Big(- \frac 2{|\S^n|} & \int_{\S^n}\log\phi \, \dsn \Big)  
\int_{\S^n}\phi \gjms (\phi)\dsn \\
& \geq 
|\S^n|^\frac{2}{1-\beta} \Big(\int_{\S^n} \phi^{1-\beta} \,\dsn\Big)^\frac{2}{\beta-1} \int_{\S^n}\phi \gjms (\phi)\dsn.
\end{align*}
With help of \eqref{eq-Sobolev}$_\beta$ we quickly obtain the inequality \eqref{eq-sSobolev-limit}. Hence we have the second $\longrightarrow$ from the left. (We should point out that the above argument works for any $\beta>1$ as  long as \eqref{eq-Sobolev}$_\beta$ is available. In particular, it works for $\beta = 2n+1$. However, we intend to keep $\beta < 2n+1$ since we want to show that the limiting case can be derived from the subcritical case.)

\noindent
\textbf{Proof of \eqref{eq-sSobolev-limit} $\longrightarrow$ \eqref{eq-Sobolev}$_\alpha$ with $\alpha \in (0,1)$}. Let us now consider arbitrary but fixed $\alpha \in (0,1)$ and we wish to derive \eqref{eq-Sobolev}$_\alpha$ from \eqref{eq-sSobolev-limit}. Indeed, still by Jensen's integral inequality \eqref{eq-Jensen} applied for $\psi = \phi^\gamma$, we obtain
\begin{equation}\label{eq-sSobolev-limit-3}
\exp \Big( \frac 2{|\S^n|}\int_{\S^n}\log  \phi \, \dsn \Big) 
\leq \Big( \frac 1{|\S^n|}\int_{\S^n}  \phi^\gamma \, \dsn\Big)^{2/\gamma}
\end{equation} 
for any $\gamma \in \R$. In \eqref{eq-sSobolev-limit-3} we choose $\gamma = 1-\alpha$ and reverse the resulting inequality to get
\[
\exp \Big( -\frac 2{|\S^n|}\int_{\S^n}\log  \phi \, \dsn \Big) 
\geq \Big( \frac 1{|\S^n|}\int_{\S^n}  \phi^{1-\alpha} \, \dsn\Big)^\frac{2}{\alpha -1}.
\]
Combining this with \eqref{eq-convention} gives
\begin{align*}
\Big(\int_{\S^n} \phi^{1-\alpha} \, & \dsn\Big)^\frac{2}{\alpha-1}
\int_{\S^n}\phi \gjms (\phi)\dsn \\
& \geq
|\S^n|^\frac{2}{\alpha-1} \exp \Big(- \frac 2{|\S^n|}\int_{\S^n}\log\phi \, \dsn \Big) 
\int_{\S^n}\phi \gjms (\phi)\dsn.
\end{align*}
With the help of \eqref{eq-sSobolev-limit} we are now able to obtain the inequality \eqref{eq-Sobolev}$_\alpha$. This establishes the last $\longrightarrow$ from the left, hence completes our proof.
\end{proof}

Again we should point out that the above argument for the last $\longrightarrow$ from the left works for any $\alpha<1$, namely we have the following sharp inequality
\[ 
\fint_{\S^n} \phi \gjms (\phi) \dsn
\geq \frac{\Gamma (n/2 + m)}{\Gamma (n/2 - m )}  \Big( \fint_{\S^n} \phi^\gamma \dsn \Big)^{2/\gamma}
\]
for any $\gamma :=1-\alpha >0$. This makes sense because $\Gamma (n/2 + m)/\Gamma (n/2 - m )<0$.

In the final discussion, we show that we can actually compare the two inequalities \eqref{eq-Sobolev}$_\beta$ with $\beta \in (1,2n+1)$ and \eqref{eq-Sobolev}$_\alpha$ with $\alpha \in (0,1)$ without using the limiting inequality \eqref{eq-sSobolev-limit}. Indeed, by decomposing the constant $1$ as 
\[
1 = \phi^\frac{(1-\alpha)(\beta-1)}{\beta - \alpha} \phi^{-\frac{(1-\alpha)(\beta-1)}{\beta - \alpha} }
\]
and applying H\"older's inequality in the following way
\[
|\S^n| \leq \Big( \int_{\S^n} \phi^{1-\alpha}  \dsn \Big)^\frac{\beta-1}{\beta-\alpha }
\Big( \int_{\S^n} \phi^{1-\beta}  \dsn \Big)^\frac{1-\alpha}{\beta-\alpha }
\]
we arrive at
\[
\Big( \int_{\S^n} \phi^{1-\alpha} \dsn \Big)^{\frac {1}{\alpha -1}} 
\leq |\S| ^\frac{\alpha-\beta}{(\beta-1)(1-\alpha)} \Big( \int_{\S^n} \phi^{1-\beta} \dsn \Big)^{\frac 1{\beta -1}}.
\]
Combining \eqref{eq-Sobolev}$_\beta$ with \eqref{eq-convention} and the identity 
\[
\frac{\beta + 1}{\beta -1} + \frac{2(\alpha-\beta)}{(\beta-1)(1-\alpha)} = \frac{\alpha+1}{\alpha-1}
\]
gives \eqref{eq-Sobolev}$_\alpha$.


\section*{Data availability}

Data sharing is not applicable to this article as no dataset was generated or analysed during the current study.

\section*{Acknowledgments}

This work was initiated and finalized when QAN was visiting the Vietnam Institute for Advanced Study in Mathematics (VIASM) in 2021 and in 2022. QAN would like to thank VIASM for hospitality and financial support. The research of QAN is funded by Vietnam National Foundation for Science and Technology Development (NAFOSTED) under grant number 101.02-2021.24 ``\textit{On several real and complex differential operators}''. AH was partially supported by the SNSF grant no. P4P4P2-194460. Last but not least, the authors would like to thank Professor Dong Ye and an anonymous referee for their careful reading and useful comments leading to some improvement of the paper. Proposition \ref{prop-order} and the discussion near the end of section \ref{sec-SS} are actually due to them.


\end{document}